\documentclass[reqno,11pt,a4paper]{amsart}

\usepackage{amssymb, graphicx}

\numberwithin{equation}{section}

\theoremstyle{theorem}
\newtheorem{theorem}{Theorem}[section]
\newtheorem{lemma}[theorem]{Lemma}
\newtheorem{proposition}[theorem]{Proposition}

\theoremstyle{definition}
\newtheorem{definition}[theorem]{Definition}

\theoremstyle{remark}
\newtheorem{remark}[theorem]{\bf Remark}
\newtheorem{example}[theorem]{\bf Example}

\makeatletter

\def\bigtimes{\mathop{\raisebox{-1.5pt}{\LARGE $\times$}}\limits}

\def\per{.}
\def\HarvardComma{}

\newcounter{aucount}
\newif\ifedplural
\newif\ifper\pertrue

\def\au#1#2{{#1 #2}}
\def\lau#1#2{{#1 #2},}
\def\ed#1#2{\ifnum\theaucount=0\relax\fi{#1 #2}\addtocounter{aucount}{1}}
\def\led#1#2{\ifnum\theaucount=0\relax\edpluralfalse\else\edpluraltrue\fi{#1
    #2} (\editorname.),\setcounter{aucount}{0}}
\def\editorname{\ifedplural Eds\else Ed\fi}

\def\et{\ifnum\theaucount=1\else\HarvardComma\fi{} and\ }
\def\ti#1{#1,\ifper\fi\pertrue}

\def\bti{\@ifnextchar[\bbti\bbbti}
\def\bbti[#1]#2{\emph{#2}, #1,}
\def\bbbti#1{\emph{#1},}

\def\z{\@ifnextchar[\zz\zzz}
\def\zz[#1]#2#3#4#5{\perfalse\emph{#2} \textbf{#3}, #4 \ifx
  @#5@\relax\else (#5)\fi [#1]\ifper\per\fi\pertrue} 
\def\zzz#1#2#3#4{\emph{#1} \textbf{#2}, #3 \ifx @#4@\relax\else
  (#4)\fi\ifper\per\fi\pertrue}

\def\pub{\@ifstar\pubstar\pubnostar}
\def\pubnostar{\@ifnextchar[\@@pubnostar\@pubnostar}
\def\@@pubnostar[#1]#2#3#4{#2, #3, #4, #1\ifper\per\fi\pertrue}
\def\@pubnostar#1#2#3{#1, #2, #3\ifper\per\fi\pertrue}
\def\pubstar[#1]#2#3#4{\perfalse #2, #3, #4 [#1]\per\pertrue}

\begin{document}

\title[Almost repetitive Delone sets]{Dynamical properties of  \\ almost repetitive Delone sets}
\author[D.\ Frettl\"oh]{Dirk Frettl\"oh}
\address{Technische Fakult\"at,
  Universit\"at Bielefeld, Universit\"atsstra{\ss}e 25, 33501 Bielefeld, Germany}
\email{dirk.frettloeh@math.uni-bielefeld.de}

\author[C.\ Richard]{Christoph Richard}        
\address{Department f\"ur Mathematik,
  Universit\"at Erlangen-N\"urnberg,
  Cauerstra\ss{}e 11,
  91058 Erlangen, Germany}
\email{richard@math.fau.de}

\begin{abstract}
We consider the collection of uniformly discrete point sets in Euclidean space equipped with the vague topology.
For a point set in this collection, we characterise minimality of an associated dynamical system by almost repetitivity of the point set. We also provide linear versions of almost repetitivity which lead to uniquely ergodic systems. Apart from linearly repetitive point sets, examples are given by periodic point sets with almost periodic modulations, and by point sets derived from primitive substitution tilings of finite local complexity with respect to the Euclidean group with dense tile orientations.
\end{abstract}

\maketitle

%
\section{Point Sets and Dynamical Systems}\label{sec1}
%

Non--periodic point sets, which still display some regularity, are interesting objects in discrete geometry. Such sets have been intensively
studied in the context of uniformly discrete subsets $P$ of Euclidean space $M=\mathbb R^\mathsf{d}$.  A useful device in that situation 
is the \emph{hull of $P$}, i.e., the orbit closure
\begin{displaymath}
X_P:=\overline{\{tP\,| \,t\in T\}},
\end{displaymath}
where $T$ is a topological group such as $\mathbb R^\mathsf{d}$ or $E(\mathsf{d})$, the group of Euclidean motions, acting continuously on $M$ 
from the left. Here the closure is taken with respect to a suitable topology on the space of uniformly discrete subsets of $M$. With the induced
action of $T$ on $X_P$, the hull can be regarded as a topological dynamical system $(X_P,T)$.
 Regularity of $P$ is then reflected in properties of its hull such as minimality or unique ergodicity.

These properties may of course depend on the topology or on the group action.  A frequently studied topology on the space of uniformly discrete point sets 
is well adapted to point sets arising from tilings \cite{gs, Ru89, RaWo92}.  This so--called \emph{local matching topology} is generated by the metric
\begin{displaymath}
\begin{split}
    d_{LM}(P,P') &:=\min\Big\{ \tfrac{1}{\sqrt{2}}\,, 
    \inf\big\{ \varepsilon>0\,|\, \exists x,x'\in B_\varepsilon \text{~such that} \\
    & (xP) \cap B_{1/\varepsilon} = P'\cap B_{1/\varepsilon}
    \text{~and~} P \cap B_{1/\varepsilon} = (x'P')\cap B_{1/\varepsilon}
    \big\}\Big\}.
\end{split}
\end{displaymath}
Here $T=\mathbb R^\mathsf{d}$ acts on $M=\mathbb R^\mathsf{d}$ canonically, and $B_s$ is the open ball of radius $s>0$ about some fixed reference point in $M$, see e.g.~\cite{LeMo02}. One may say that two point sets are close in this topology, if they agree -- after some small translation -- on a large ball about some fixed center. 
Attention is often restricted to point sets of so--called \emph{finite local complexity} (FLC), since many point sets derived from substitution tilings such as the non--periodic Penrose tiling  \cite{Ro96} share this property.  In this context, a geometric characterisation of minimality is \emph{repetitivity} of the point set \cite[Thm~3.2]{LaPl03}, and a geometric characterisation of unique ergodicity is \emph{uniform pattern frequencies} of the point set \cite[Thm~2.7]{LeMo02}. In particular, linear repetititivity implies uniform pattern frequencies \cite[Thm~6.1]{LaPl03}, hence also unique ergodicity. The above characterisations have been extended to more general groups $T$ and to point spaces $M$ more general than Euclidean space in \cite[Prop~4.16]{Yok05} and in \cite[Prop~2.32]{MR}. 

Let us compare these results to tiling space dynamical systems of finite local complexity. There one considers the collection of all tilings built from translated prototiles, where some matching rules have to be satisfied. A local matching topology for tilings turns this collection into a compact topological space. Dynamical properties of this space with respect to the induced group action can then be studied combinatorially by exploiting how large patches are built from smaller ones. A primitive substitution of finite local complexity leads to repetitive tilings, and hence to a minimal tiling space dynamical system, see \cite{S} and \cite[Thm~5.8]{Ro04}. Moreover, such substitutions result in uniform patch frequencies, which implies unique ergodicity of the tiling space dynamical system \cite[Thm~6.1]{Ro04}. As has been shown recently \cite[Thm~3.8]{CS}, substitution matrices can be used to parametrise the simplex of invariant probability measures over the tiling space dynamical system. Whereas this indeed leads to a characterisation of unique ergodicity, it is not obvious to us how to interpret this condition geometrically.
The above properties have also been studied in the considerably more general context of fusion tilings \cite{FS2} of finite local complexity. If the tiling space dynamical system is topologically transitive, then it can be compared to a corresponding point set dynamical system. This is the case for substitutions or fusion rules that are primitive. 

In this article our main focus is on point sets rather than on tilings. We ask which point sets of \emph{infinite local complexity} may still have a minimal or uniquely ergodic hull. Previous results on point sets in this direction appear in \cite[Thm~2.6(i)]{BBG}, \cite[Thm~3(b)]{BaLe05} and \cite[Thm~3.1(c)]{LeRi07}. In order to have a compact hull, we use the vague topology, also called \emph{local rubber topology}, see \cite[Chapter~2.1]{MR} for a discussion of the historical background. It is generated by the metric
\begin{displaymath}
    d_{LR}(P,P') :=\min\Big\{ \tfrac{1}{\sqrt{2}}\,, 
    \inf\big\{ \varepsilon>0\,|\, 
    P \cap B_{1/\varepsilon}\subseteq  (P')_\varepsilon 
    \text{~~and~~} P' \cap B_{1/\varepsilon} \subseteq (P)_{\varepsilon} \big\}
    \Big\},
  \end{displaymath}
where the ``thickened'' point set $(P)_\varepsilon := \bigcup_{p\in P} B_{\varepsilon}(p)$ is the set of points in $M$ lying within distance less than $\varepsilon$ to $P$.  (Here the triangle inequality rests on $(A)_\varepsilon\cap B\subseteq (A\cap (B)_\varepsilon)_\varepsilon$.) We may say that two point sets are close in the vague topology, if they \emph{almost} agree on a large ball about some fixed center. For point sets of finite local complexity and for a transitive and proper group action, the local rubber topology equals the local matching topology, since a point set of finite local complexity is locally rigid \cite[Lemma 2.27]{MR}.

Let us summarise and discuss the results of our article. In the local rubber topology, even without finite local complexity, minimality of the hull is equivalent to \emph{almost repetitivity} of the point set, see Theorem~\ref{charmin} below. Whereas a similar statement for $M=T=\mathbb R^\mathsf{d}$ already appears in \cite[Thm~2.6(i)]{BBG}, we provide a proof within the more general setting of \cite{MR}. 
Concerning unique ergodicity, we restrict to Euclidean space $M=\mathbb R^\mathsf{d}$, as our proofs crucially rely on box decompositions. Even in that situation, we cannot expect a geometric characterisation in terms of suitable pattern frequencies without further assumptions on the point set.  For uniformly discrete point sets, we show that \emph{almost linear repetitivity} ensures unique ergodicity with respect to $T=\mathbb R^\mathsf{d}$ and $T=E(\mathsf{d})$ in Theorem~\ref{thm:alruni}. In the case of finite local complexity, this implies the known result that linear repetitivity ensures unique ergodicity w.r.t.~$T=\mathbb R^\mathsf{d}$, see \cite[Thm~6.1]{LaPl03} and \cite[Cor~4.6]{DaLe01}. Our proof adapts the reasoning of \cite{DaLe01} and does not resort to pattern frequencies. Examples are periodic point sets with almost periodic modulations, see Example~\ref{ex:lap}.  As an extension, we show in Theorem~\ref{thm:alwruni} and Theorem~\ref{thm:alwrunied} that \emph{almost linear wiggle--repetitivity} also ensures unique ergodicity with respect to $T=\mathbb R^\mathsf{d}$ and $T=E(\mathsf{d})$. This is the main result of our article.  

As an application of our general results above, we consider primitive substitution tiling spaces of finite local complexity w.r.t.~the Euclidean group with dense tile orientations. Whereas this comprises tiling spaces of infinite local complexity w.r.t.~$\mathbb R^\mathsf{d}$ such as those of pinwheel tilings \cite{R95b} and variants \cite{Sa1, CoRa98, F08}, it does not include spaces of tilings of infinite local complexity with fault lines, compare \cite{FS}.  We show in Theorem~\ref{thm:pin-wigrep} that such tilings are linearly wiggle-repetitive, which then implies that the tiling space dynamical system with $\mathbb R^\mathsf{d}$-action is minimal and uniquely ergodic. This is stated in Theorem~\ref{unicirc}, which combines  and extends the well-known results \cite{R95b} and \cite[Thm~3.1]{So97}. Note that minimality already follows from a recent result on fusion tilings  of infinite local complexity \cite[Prop~3.1]{FS3}, see also \cite[Prop~3.2]{FS2}. We would like to remark that our approach is complementary. Whereas the previous approaches use combinatorial properties of the substitution to infer dynamical properties directly, we follow a geometric viewpoint by extracting those repetitivity properties which cause minimality and unique ergodicity. In particular, we do not resort to patch frequencies.

In the context of general fusion tiling dynamical systems of infinite local complexity, a primitive fusion rule may lead to almost repetitive tilings. A geometric characterisation of unique ergodicity in this context seems beyond the scope of this article. As a first step, one may study uniquely ergodic systems some of whose members fail to be almost linearly wiggle-repetitive.

In order to review old and motivate new notions of repetitivity by examples, we discuss substitution tilings of finite local complexity with respect to the Euclidean group with dense tile orientations in the following section. In Section 3, we study almost repetitivity for point sets in the general setting of \cite{MR}.  Section 4 specialises to Euclidean space and discusses implications of almost linear repetitivity with respect to $T=\mathbb R^\mathsf{d}$. Section 5 is devoted to a study of almost wiggle--repetitivity within the Euclidean setting. In the last section, our results on point set dynamical systems are applied to the tiling spaces of Section~\ref{chap:subs}. 

\section{Substitution Tilings with Dense Tile Orientations}
\label{chap:subs}

We consider self-similar substitution tilings in Euclidean space $\mathbb R^\mathsf{d}$ in this section, where $\mathsf{d}\in\mathbb N$. A corresponding setup, where translated tiles are identified, has been worked out in detail in \cite{So97,S,Ro04,F08}. As in \cite{RaWo92, R95b}, compare also \cite{FS2}, we want to identify tiles that are equal up to Euclidean motion.  Since this requires some adaption of the above setup, we give a detailed presentation.

Let a \emph{tile} $S$ be a subset of Euclidean space homeomorphic to a closed unit ball of full dimension.  A \emph{(tile) packing} is a (countable) collection of tiles having mutually disjoint interior. The \emph{support} of a  packing $\mathcal C$ is the set $\mathrm{supp}(\mathcal C):=\bigcup_{S\in\mathcal C} S$. A \emph{tiling} $\mathcal{T}$ is a packing which covers Euclidean space, i.e., $\mathrm{supp}(\mathcal{T})=\mathbb R^\mathsf{d}$.

A \emph{patch} is a finite packing. If a patch $\mathcal P$ is contained in a packing $\mathcal C$, we say that $\mathcal P$ is a \emph{patch of $\mathcal C$}. 
A patch $\mathcal P$ of $\mathcal C$ is called the \emph{$s$-patch of $\mathcal C$ centered in $x$}, if $\mathcal P=\{ S \in \mathcal C \, | \, S \subseteq \overline{B}_s(x) \}$, with $B_s(x)$ the open ball of radius $s$ about $x$. We also speak of an \emph{$s$-patch of $\mathcal C$} in that situation.

Often tilings are built from equivalent copies of finitely many fundamental tiles. To describe this, we fix a subgroup $G$ of $E(\mathsf{d})$ containing $\mathbb R^\mathsf{d}$. Two tiles $S',S$ are called \emph{$G$-equivalent} if $S'=gS$ for some $g\in G$. Likewise, two packings $\mathcal C',\mathcal C$ are called $G$-equivalent if $\mathcal C'=g\mathcal C$ for some $g\in G$, where $g\mathcal C=\{gS\,|\,S\in\mathcal C\}$.  We fix a non-empty finite set $\mathcal{F}$ of mutually $G$-inequivalent tiles, which are called \emph{prototiles}\footnote
{Sometimes it is necessary to distinguish tile types even though they are $G$-equivalent. This can be done by adding a label to the tiles \cite{So97}. Then one needs to write ``tile $(S,i)$'' rather than just ``tile $S$''. For the sake of clarity we only consider $G$-inequivalent prototiles here. 
}. 
A packing is called an \emph{$(\mathcal{F},G)$-packing} if every tile in the packing is $G$-equivalent to some prototile of $\mathcal F$. The collection of all $(\mathcal{F},G)$-packings is denoted by $\mathcal C_{(\mathcal{F},G)}$.

An important class of tilings are substitution tilings. These are generated from prototiles by some inflate-and-dissect rule as in Figure \ref{fig:pin}. 
\begin{figure}
\includegraphics[width=90mm]{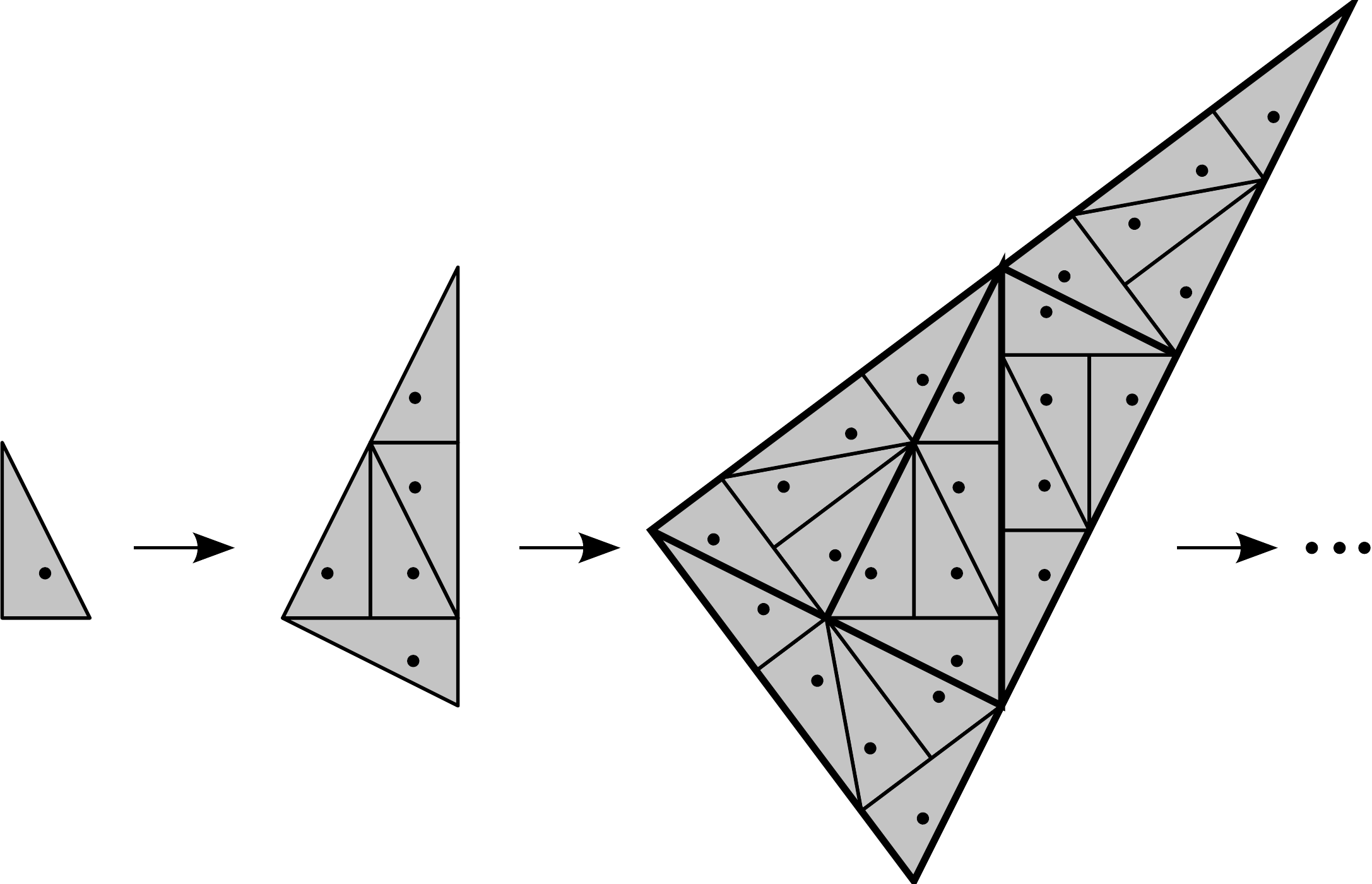} 
\caption{A substitution with only one prototile. The point in the first triangle is the rotation center. Iterating the
  substitution yields a unique tiling of the plane, a pinwheel tiling \cite{R95b}. It is fixed under the
  substitution. \label{fig:pin}}
\end{figure}
A \emph{(self-similar tile) substitution} $\sigma$ on $(\mathcal F,G)$ with prototile set $\mathcal{F}=\{S_1,\ldots,S_m\}$ is given by $(\mathcal{F},G)$-patches $\sigma(\{S_1\}),\ldots,\sigma(\{S_m\})$,
some $\lambda>1$ and some rotation $r_0\in O(\mathsf{d})$ about the origin, such that 
\begin{displaymath}
\mathrm{supp}(\sigma(\{S_i\}))=r_0\cdot\lambda S_i
\end{displaymath}
for every $i$. Whereas in $\mathsf{d}\le \mathsf{2}$ the rotation $r_0$ may be arbitrary, we require $r_0$ to be the identity in $\mathsf{d}>\mathsf{2}$, see also below.
The factor $\lambda$ is called the {\em substitution factor}. Let $n_{ij}\in\mathbb N_0$ denote the number of 
tiles $G$-equivalent to $S_i$ in $\sigma(\{S_j\})$. The matrix $M_{\sigma} = 
(n_{ij})_{1 \le i,j \le m}$ is called the {\em substitution matrix}. The
substitution $\sigma$ is called {\em primitive} if $M_{\sigma}$ is primitive, i.e., if some power of $M_{\sigma}$ is positive. 

The above definition of $\sigma$ extends naturally to a map on the set $\mathcal C_{(\mathcal{F},G)}$ of $(\mathcal{F},G)$-packings into itself, which we denote by $\sigma$ again: Let $S$ be a tile $G$-equivalent to some prototile $S_j$. Then $S=x+r\cdot S_j$ for some translation $x$ and some rotation or rotation-reflection $r$ about the origin. The substitution procedure of Figure \ref{fig:pin} is implemented by defining the $(\mathcal{F},G)$-patch
\begin{displaymath}
\sigma(\{S\}):=\lambda r_0 \cdot x+r\cdot\sigma(\{S_j\}).
\end{displaymath}
We also have $\mathrm{supp}(\sigma(\{S\}))=r_0\cdot\lambda\, S$, where we use  commutativity of the group $O(\mathsf{d})$ for $\mathsf{d}\le\mathsf{2}$. For any $(\mathcal{F},G)$-packing $\mathcal C$ we then define $\sigma(\mathcal C):=\bigcup_{S\in\mathcal C}\sigma(S)$. Using $\mathrm{supp}(\sigma(\mathcal C))=r_0\cdot\lambda\,\mathrm{supp}(\mathcal C)$, one shows that $\sigma(\mathcal C)$ is indeed an $(\mathcal{F},G)$-packing. 
Hence the map $\sigma:\mathcal C_{(\mathcal{F},G)}\to\mathcal C_{(\mathcal{F},G)}$ thus obtained is well defined. For any two $(\mathcal{F},G)$-packings $\mathcal C'$ and $\mathcal C$, it can be shown that $\mathcal C'$ and $\mathcal C$ are $G$-equivalent if and only if $\sigma(\mathcal C')$ and $\sigma(\mathcal C)$ are $G$-equivalent.

For $k\in\mathbb N_0$, we also consider $k$-fold iterates of $\sigma$, with $\sigma^0$ the identity. For a tile $S$ which is $G$-equivalent to some prototile, we call $\sigma^k(\{S\})$ the \emph{$k^\text{th}$-order supertile of $S$}. We say that an $(\mathcal{F},G)$-packing \emph{has a $k^\text{th}$-order supertile of type $j$} if a $G$-equivalent copy of the supertile $\sigma^k(\{S_j\})$ is contained in the packing. 

\begin{definition} \label{def:substil}
Let $\sigma$ be a substitution on $(\mathcal F,G)$. An $(\mathcal{F},G)$-patch is \emph{legal} if it is contained in some supertile. A {\em substitution tiling} is a tiling such that every of its patches is legal. The set $X_\sigma=X_{(\sigma,\mathcal F,G)}$ of all substitution tilings is called the \emph{tiling space of $(\sigma, \mathcal F,G)$}. 
\end{definition}

\begin{remark}\label{legal}
(i) Our definition of tiling space is adapted to deal with substitutions of finite local complexity with respect to the Euclidean group as in Definition~\ref{def:FLC} below. It differs from that in \cite{FS3}, which is more natural when also dealing with tilings of infinite local complexity with respect to the Euclidean group. 

(ii) If a patch $\mathcal{P}$ is legal, then also $\sigma(\mathcal{P})$ is legal. Note that we allow for supertiles of order zero in the definition of a legal patch, in contrast to \cite{S}. For primitive substitutions, both definitions are equivalent.  

(iii) To analyse whether $X_\sigma$ is nonempty, consider \cite[p.~229]{lup} any sequence $j_1,\ldots, j_{m+1}$ of tile types, such that for $k\in\{1,\ldots,m\}$ the supertile $\sigma(\{S_{j_{k}}\})$ contains a tile of type $j_{k+1}$. Since there are only $m$ types, some type $j$ occurs twice in the sequence, and hence some supertile $\sigma^k(\{S_j\})$ has a tile $S$ of type $j$. Write $S_j=g\cdot S$, with $g=(x,r)$ and $g\cdot S=x+r\cdot S$ and consider the sequence $((g\sigma^k)^n(\{S_j\}))_{n\in\mathbb N_0}$ of legal patches. If $S$ is contained in the interior of $\mathrm{supp}(\sigma^k(\{S_j\}))$, the members of this sequence finally coincide on arbitrarily large patches about the origin. In that case the sequence gives rise to a legal tiling, which is fixed under $\sigma^k$. As the example Figure \ref{fig:pin} shows, the latter condition can be somewhat relaxed. 
\end{remark}

\begin{lemma} \label{lem:suptile}
Let $\sigma$ be a substitution on $(\mathcal{F},G)$. Fix $k\in\mathbb N_0$ and choose $r_k>0$ such that every $k^\text{th}$-order supertile support is contained in a ball of radius $r_k$. Then for any tiling $\mathcal{T}\in X_\sigma$ the following hold.
\begin{itemize}
\item[i)] $\mathcal{T}$  can be partitioned into $k^\text{th}$-order supertiles.
\item[ii)] Every $2r_k$-patch of $\mathcal{T}$ contains some $k^\text{th}$-order supertile of $\mathcal{T}$.
\end{itemize}
\end{lemma}

\begin{remark}
As the following proof shows, a partition satisfying i) may not be uniquely determined, compare \cite[Thm~1.1]{S}.
\end{remark}

\begin{proof}
i). W.l.o.g.~fix $k\in\mathbb N$. Choose $n \in \mathbb N$ sufficiently large such that the $n$-patch $\mathcal C_{n}$ centered at the origin can not be patch of a supertile of order less than $k$. Thus $\mathcal C_{n}$ is, by definition, contained in some supertile of order at least $k$. 
But this supertile can, by definition, be partitioned into $k^\text{th}$-order supertiles. This induces a  partition $\pi({\mathcal C}_{n})$ of $\mathcal C_{n}$ into $k^\text{th}$-order supertiles of $\mathcal{T}$ and a \emph{boundary patch}, i.e., a patch which does not contain any $k^\text{th}$-order supertile. Now consider a sequence $(\pi({\mathcal C}_n))_{n}$ of such partitions. Since there are only finitely many partitions of the patch $\mathcal C_n$, one may choose a subsequence $(\pi({\mathcal C}_{n_\ell}))_{\ell}$ of $(\pi({\mathcal C}_n))_{n}$ of \emph{consistent partitions}, i.e., every $k^\text{th}$-order supertile in $\pi({\mathcal C}_{n_\ell})$ is also in $\pi({\mathcal C}_{n_m})$, if $m>\ell$. Hence $\mathcal{T}$ can be partitioned into $k^\text{th}$-order supertiles.

ii). By i), we can choose a partition of $\mathcal{T}$ into $k^\text{th}$-order supertiles. Choose  $\{x_i\,|\,i\in\mathbb N\}$ such that every $k^\text{th}$-order supertile support is contained in some ball $B_{r_k}(x_i)$. Then $(\overline{B}_{r_k}(x_i))_{i\in\mathbb N}$ covers Euclidean space, and every $r_k$-patch of $\mathcal{T}$ centered at $x_i$ contains some $k^\text{th}$-order supertile of $\mathcal{T}$. Take arbitrary $x\in\mathbb R^\mathsf{d}$. Then there is some $x_i$ such that 
$d(x,x_i)\le r_k$ and hence $\overline{B}_{r_k}(x_i)\subseteq \overline{B}_{2r_k}(x)$. Thus the $2r_k$-patch of $\mathcal{T}$ centered at $x$ contains the $r_k$-patch of $\mathcal{T}$ centered at $x_i$ and hence contains some $k^\text{th}$-order supertile of $\mathcal{T}$.
\end{proof}

\begin{lemma}\label{lem:suptile2}
Let $\sigma$ be a primitive substitution on $(\mathcal F,G)$. 
Then the following hold.
\begin{itemize} 
\item[i)] For every $k \in\mathbb N_0$ there is $K=K(k)\in\mathbb N_0$ such that
every supertile of order at least $K$ has every type of $\ell^\text{th}$-order supertile, for any $\ell\le k$.  
\item[ii)] Every tiling in $X_{\sigma}$ has all types of supertile of any order. 
\item[iii)] Every tiling in $X_\sigma$ contains an equivalent copy of every legal patch.

\end{itemize}

\end{lemma}

\begin{proof}

i). Since $\sigma$ is primitive, there is $n$ such that $(M_{\sigma})^n$ is positive, i.e., every $n^\text{th}$-order supertile contains an equivalent copy of each prototile. Fix $k\in\mathbb N_0$. Then for any $\ell\in\{0,\ldots,k\}$,  the matrix $(M_{\sigma})^{n+k-\ell}$ is positive as well, thus every $(k+n)^\text{th}$-order supertile contains an equivalent copy of every $\ell^\text{th}$-order supertile. Hence the statement follows with $K(k)=k+n$.

ii). Take $\mathcal{T} \in X_{\sigma}$ and consider arbitrary $k$. Due to Lemma~\ref{lem:suptile} i), one may partition $\mathcal T$ can into $K(k)^\text{th}$-order supertiles, with $K(k)$ as in i). Hence $\mathcal T$ has $k^\text{th}$-order supertiles of all types.

iii). This is an immediate consequence of ii).
\end{proof}

The following definition is useful when comparing different types of repetitivity.

\begin{definition}\label{def:FLC}
Let $G$ be a subgroup of $E(\mathsf{d})$ containing $\mathbb R^\mathsf{d}$.

\begin{itemize}
\item[i)] A tiling $\mathcal{T}$ has FLC w.r.t.~$G$, if for every $r>0$ the number of $G$-inequivalent $r$-patches in $\mathcal{T}$ is finite.
\item[ii)] Let $\sigma$ be a substitution on $(\mathcal F,G)$. Then $\sigma$ has \emph{finite local complexity} (FLC), if for every $r>0$ the number of $G$-inequivalent $r$-patches 
from all supertiles is finite.
\end{itemize}
\end{definition}

\begin{remark}
Fix $G=E(\mathsf{d})$. For a tiling $\mathcal{T}$ of polygons with a finite number of prototiles, it is not 
hard to see that $\mathcal{T}$ is of FLC, if the tiles in $\mathcal{T}$ meet full-face to full-face. See \cite{FR} for further criteria.  The  face-to-face criterion can 
be applied to pinwheel tilings, when viewing the triangles, after subdivision of their medium edge, as degenerate quadrangles.   Note that the pinwheel tilings do not have FLC w.r.t.~$G=\mathbb R^\mathsf{2}$. Examples in $\mathsf{d}=\mathsf{3}$ are quaquaversal tilings \cite{CoRa98}.
\end{remark}

\begin{definition}
Let $G$ be a subgroup of $E(\mathsf{d})$ containing $\mathbb R^\mathsf{d}$. Then a tiling $\mathcal{T}$ of Euclidean space is called

\begin{itemize}
\item[i)] \emph{weakly repetitive w.r.t.~$G$}, if for every patch $\mathcal{P}$ of $\mathcal{T}$ there exists $R=R(\mathcal{P})>0$ such that every $R$-patch of $\mathcal{T}$ contains a $G$-equivalent copy of $\mathcal{P}$.

\item[ii)] \emph{repetitive w.r.t.~$G$}, if for every $r>0$ there exists $R=R(r)>0$ such that every $R$-patch of $\mathcal{T}$ contains a $G$-equivalent copy of every $r$-patch of $\mathcal{T}$.

\item[iii)] \emph{linearly repetitive w.r.t.~$G$}, if $\mathcal{T}$ is repetitive and if one can choose $R(r)=\mathcal O(r)$ as $r\to\infty$.
\end{itemize}

\end{definition}

\begin{remark}\label{rem:rep}
(i)  It can be shown that weak repetitivity and FLC is equivalent to repetitivity, compare the proof of Proposition~\ref{thm:psubs-rep}.

(ii) The above notions of repetitivity appear in different forms under
different names in the literature: Repetitivity (and also its linear variant) has already been studied in the context of symbolic dynamics, where it is called recurrence \cite{MH38, MH40}.  In \cite{gs} and \cite{RaWo92}, weak repetitivity is called the local isomorphism property, where \cite{gs} prove weak repetitivity for Penrose tilings by proving linear repetitivity.
In \cite{bt}, weakly repetitive is called recurrent.
In \cite{thu}, repetitive is called quasi-homogeneous, in \cite{lup}
it is called quasiperiodic.  In \cite{S} and \cite{BBG}, weak repetitivity is called repetitivity, 
and a linear version of weak repetitivity is called strong repetitivity in \cite{S}.
The term repetitivity was possibly coined
by Danzer \cite{dan}, where it was used for linear repetitivity. 
In \cite{LaPl03}, a clear distinction is made between the different kinds of repetitivity, 
using the terminology above (on point sets which may arise from tilings). We will stick to that terminology throughout this article.
\end{remark}

The proof of the following result is standard, see e.g.~\cite{S} for $G=\mathbb R^\mathsf{d}$.

\begin{proposition} \label{thm:psubs-rep}
Let $\sigma$ be a primitive substitution on $(\mathcal F,G)$. Then for any tiling $\mathcal{T} \in X_\sigma$ the following hold.
\begin{itemize} 
\item[i)] $\mathcal{T}$ is weakly repetitive w.r.t.~$G$. 
\item[ii)] If $\sigma$ has also FLC, then $\mathcal{T}$ is repetitive w.r.t.~$G$.
\end{itemize}
\end{proposition}

\begin{proof} 

i). Let $\mathcal{P}$ be a patch in $\mathcal{T}$. Then $\mathcal{P}$ is contained in some $k^\text{th}$-order supertile.  By Lemma \ref{lem:suptile2} i), there is $K=K(k)\in\mathbb N$ such that every $K^\text{th}$-order supertile contains a $G$-equivalent copy of every supertile up to $k^\text{th}$ order. By Lemma \ref{lem:suptile} ii), we can choose $R=R(\mathcal{P})>0$ such that every $R$-patch of $\mathcal{T}$ contains some $K^\text{th}$-order supertile. Hence it also contains a $G$-equivalent copy of $\mathcal{P}$.

ii).  Let $r>0$ be given. Choose a collection of mutually $G$-inequivalent $r$-patches of $\mathcal{T}$ of maximal cardinality. This collection $\{\mathcal{P}_i\,|\, i\in I\}$ is finite due to FLC. Choose $R=R(r)$ as the maximum of $R(\mathcal{P}_i)$ over $I$ in i). Then every $R$-patch of $\mathcal{T}$ contains a $G$-equivalent copy of any $r$-patch of $\mathcal{T}$. Thus $\mathcal{T}$ is repetitive w.r.t.~$G$.
\end{proof}

Linear repetitivity is proven for the Penrose tiling and $G=\mathbb R^\mathsf{2}$ in \cite{gs}. A version for primitive substitution tilings and $G=\mathbb R^\mathsf{d}$ can be deduced from \cite[Lemma~2.3]{S}. In order to prepare for Theorem~\ref{thm:pin-wigrep}, we give a version of that proof which also works in our situation. It uses supertile coronae, compare \cite{S}. Any $\ell^\text{th}$-order supertile $\mathcal S$ of a substitution has a canonical partition $\mathcal S_{(k)}$ into $k^\text{th}$-order supertiles, where $0\le k\le \ell$. For  any $k^\text{th}$-order supertile $\mathcal C_k$ in $\mathcal S_{(k)}$, consider the patch  $[\mathcal C_k]$ obtained 
from the union of all elements in $\mathcal S_{(k)}$ whose support has non-empty intersection with $\mathrm{supp}(\mathcal C_k)$. We call $[\mathcal C_k]$
the \emph{supertile corona of $\mathcal C_k$ (with respect to the supertile $\mathcal S$}).

\begin{remark}
The following properties of supertile coronae are essential for the proof below.
For any FLC substitution, there exists $\varrho_0>0$ such that for any two tiles $S,S'$ in any legal patch the condition $d(x,x')<\varrho_0$ for some $x\in S$ and some $x'\in S'$ implies that $S\cap S'\ne\varnothing$. (Otherwise pairs of non-intersecting tiles will get arbitrarily close in legal patches, which contradicts FLC.) For any ball 
$B_{\varrho_0}(x)$, consider any supertile $\mathcal S$ covering $B_{\varrho_0}(x)$, and take any tile $S$ of $\mathcal S$ containing $x$.
Then, by the above reasoning, $B_{\varrho_0}(x)$ is covered by the \emph{tile corona} $[\{S\}]$ of $\{S\}$.
More generally, consider any ball $B_{\varrho_k}(x)$, where $\varrho_k=\lambda^k\varrho_0$. Let $\mathcal S$ be any supertile of order $\ell\ge k$ which covers $B_{\varrho_k}(x)$. Let $\mathcal C_k$ be any $k^\text{th}$-order supertile in $\mathcal S_{(k)}$ whose support contains $x$. Then its supertile corona $[\mathcal C_k]$ covers $B_{\varrho_k}(x)$. This holds since, after $k$-fold canonical deflation of the supertile $\mathcal S$, the corresponding tile corona $[\{S\}]$, where $\sigma^k(\{S\})=\mathcal C_k$, is contained in some supertile by definition. Hence the $\varrho_0$-ball corresponding to $B_{\varrho_k}(x)$ is covered by $[\{S\}]$ by the above argument. After applying $\sigma^k$, we see that $B_{\varrho_k}(x)$ is indeed covered by $[\mathcal C_k]$.
\end{remark}

\begin{proposition} \label{thm:psubs-linrep}
Let $\sigma$ be a primitive substitution on $(\mathcal F,G)$ of FLC. 
Then every tiling in $X_{\sigma}$ is linearly repetitive w.r.t.~$G$.
\end{proposition}

\begin{proof} 
Consider an arbitrary tiling $\mathcal T\in X_\sigma$. Since $\mathcal{T}$ is repetitive w.r.t.~$G$ by Proposition~\ref{thm:psubs-rep} ii), there exists $R_0>0$ such that every $R_0$-patch of $\mathcal{T}$ contains an equivalent copy of every tile corona in $\mathcal{T}$. Then, by Lemma~\ref{lem:suptile2} iii), every legal $R_0$-patch contains an equivalent copy of every legal tile corona. As in the previous remark, we can conclude that every $R_k$-patch of $\mathcal{T}$, where $R_k=\lambda^k R_0$, contains an equivalent copy of every $k^\text{th}$-order supertile corona of $\mathcal{T}$. To see this, consider any ball $B_{R_k}(x)$. The corresponding $R_k$-patch of $\mathcal T$ centered in $x$ is, by definition, contained in some supertile $\mathcal S$, which we may w.l.o.g.~assume to be of order not less than $k$, since $\mathcal T$ is a tiling.  Now take the union $\mathcal C$ of all supertiles in $\mathcal S_{(k)}$ whose support has non-empty intersection with $B_{R_k}(x)$.  Then after $k$-fold canonical deflation of the supertile $\mathcal S$, the corresponding patch $\mathcal P$, where $\sigma^k(\mathcal P)=\mathcal C$,  is legal by definition, and its support contains a ball of radius $R_0$. But the corresponding $R_0$-patch is also legal and contains an equivalent copy of every legal tile corona. Hence the claim is seen to be true after applying $\sigma^k$.

Now consider any $\varrho_k$-patch of $\mathcal{T}$, where $\varrho_k=\lambda^k\varrho_0$ as in the previous remark. Since this patch is by definition contained in some supertile (w.l.o.g.~of order not less than $k$), it must be contained in some $k^\text{th}$-order supertile corona. Let $r>0$ be given. Take the smallest $k$ such that $\varrho_k>r$. Then by the arguments above, every $R_k$-patch of $\mathcal{T}$ contains an equivalent copy of every $k^\text{th}$-order supertile corona of $\mathcal{T}$, hence of every $\varrho_k$-patch of $\mathcal T$, and hence of every $r$-patch of $\mathcal{T}$. Since
\begin{displaymath}
R_k=\lambda^k R_0\le\lambda\frac{R_0}{\varrho_0}r,
\end{displaymath}
we may choose $R(r)=c\cdot r$, with $c=\lambda R_0/\varrho_0$, and linear repetitivity is shown.
\end{proof}

Choose a metric on $O(\mathsf{d})$ generating the standard topology.
The following property is crucial to infer dynamical properties of the tiling space.

\begin{definition}
A tiling $\mathcal T$ of Euclidean space is called
\begin{itemize}
\item[i)] \emph{wiggle--repetitive}, if for every $r>0$ and for every $\varepsilon>0$ there exists $R=R(r,\varepsilon)>0$ such that every $R$-patch of $\mathcal{T}$ contains an $E(\mathsf{d})$-equivalent copy of every $r$-patch of $\mathcal{T}$, with a corresponding rotation of distance less than $\varepsilon$ to the identity. 
\item[ii)] \emph{linearly wiggle--repetitive}, if it is wiggle--repetitive, and if for every $\varepsilon>0$ one can choose $R(r,\varepsilon)=\mathcal O(r)$ as $r\to\infty$, 
where the $\mathcal O$--constant may depend on $\varepsilon$.
\end{itemize}
\end{definition}

A weak version of wiggle--repetitivity is studied in \cite[Thm~6.3]{F08}. In order to strengthen that result, we need
some terminology. For a tile $S$, every $r\in O(\mathsf{d})$ such that $S=x+r\cdot S_j$ is called an \emph{orientation} of $S$. 
Call $D\subseteq H\subseteq  O(\mathsf{d})$ to be \emph{$\varepsilon$-dense
in $H$}, if every ball in $H$ of radius $\varepsilon$ has non-empty intersection with $D$. 
\begin{definition}
Fix $G=\mathbb R^\mathsf{d}\rtimes H$, with $H$ a subgroup of $O(\mathsf{d})$. Let $\sigma$ be a substitution on $(\mathcal F,G)$. We say that $\sigma$ has \emph{dense tile orientations (DTO) in $H$} if, for every $\varepsilon>0$ and for every $j$, there is some supertile with $\varepsilon$-dense orientations of tiles of type $j$ in $H$.
\end{definition}

A simple example is the chair tiling  \cite[Fig~10.1.5]{gs} with discrete $H=D_4$, the dihedral group of order eight. More interesting examples concern 
non--discrete subgroups of $O(\mathsf{d})$. For convenience we give the following handy characterisation of DTO in 
$O(\mathsf{2})$, which can be inferred from \cite[Prop~3.4]{F08}.

\begin{lemma}
Let $\sigma$ be a primitive substitution on $(\mathcal F, E(\mathsf{2}))$. Then the following are equivalent.
\begin{itemize}
\item[i)] $\sigma$ has dense tile orientations in $O(\mathsf{2})$.
\item[ii)] There exists a supertile containing two tiles of the same type, which are rotated against each other by an irrational angle.
\qed
\end{itemize}
\end{lemma}

\begin{example}
The pinwheel tiling has DTO in $O(\mathsf{2})$, as can be seen in the $2^\text{nd}$-order supertile using the above criterion.
Other planar examples appear in \cite{Sa1,F08}. Quaquaversal tilings have DTO in $SO(\mathsf{3})$, as follows from \cite[Thm.~1]{CoRa98}.
\end{example}

The following theorem gives a sufficient condition for wiggle--repetitivity.

\begin{theorem} \label{thm:pin-wigrep}
Fix $G=\mathbb R^\mathsf{d}\rtimes H$, with $H$ a subgroup of $O(\mathsf{d})$. 
Let  $\sigma$ be a primitive substitution on $(\mathcal{F},G)$ with FLC w.r.t.~$G$ and DTO w.r.t.~$H$. 
Then every tiling in $X_\sigma$ is linearly wiggle--repetitive.
\end{theorem}

\begin{proof}
Take arbitrary $\mathcal{T}\in X_\sigma$. First we show that $\mathcal T$ is wiggle--repetitive. 

Fix arbitrary $\varepsilon>0$. Since $\sigma$ has dense tile orientations in $H$, there is $k=k(\varepsilon)$
such that some $k^\text{th}$-order supertile has type 1 tiles in $\varepsilon$-dense orientations in $H$. Define $\ell(\varepsilon):=k+n$, with $n$ such that $(M_\sigma)^n$ is positive. Then every $\ell^\text{th}$-order supertile has type 1 tiles in $\varepsilon$-dense orientations in $H$.

Now fix arbitrary $r>0$.
Since $\mathcal T$ is repetitive w.r.t.~$G$ by Proposition~\ref{thm:psubs-rep} ii), we can 
choose $R'=R'(r)$ such that every $R'$-patch of $\mathcal{T}$ contains a $G$-equivalent copy 
of any $r$-patch of $\mathcal{T}$. Hence, by Lemma~\ref{lem:suptile2} iii), every legal $R'$-patch contains 
a $G$-equivalent copy of any legal $r$-patch. Choose $N=N(R')$ such that every 
$N^\text{th}$-order supertile contains some $R'$-patch.
Then every $N^\text{th}$-order supertile contains a $G$-equivalent copy of any legal $r$-patch.

Combining the above two arguments, we conclude that every $(N+\ell)^\text{th}$-order supertile of contains equivalent copies in $\varepsilon$-dense orientations in $H$ of any legal $r$-patch. (Here the orientation of a patch is defined as the orientation of a reference tile in the patch.)  By Lemma \ref{lem:suptile} ii), every $2r_{N+\ell}$-patch of $\mathcal T$ contains some $(N+\ell)^\text{th}$-order supertile. It follows that $\mathcal T$ is wiggle--repetitive, with $R=R(r,\varepsilon)=2r_{N+\ell}$.

To show that $\mathcal T$ is also linearly wiggle--repetitive, we 
explicate the constants. Fix arbitrary $\varepsilon>0$ and, without loss of generality, arbitrary $r>1$.
By linear repetitivity w.r.t.~$G$, see Proposition~\ref{thm:psubs-linrep}, we can choose $R'(r) = L r$ for some $L>0$. 
Let $s_k=\lambda^ks_0>0$ be such that every $k^\text{th}$-order supertile support contains some ball of radius $s_k$.
We can choose $N=N(R')$ as the smallest integer such that $s_N > R'$.  Since we may choose $r_k=\lambda^k r_0>0$, we have
\begin{displaymath}
2r_{N+\ell}=2\lambda^{N+\ell}r_0\le 2\lambda^{\ell+1}\frac{r_0}{s_0}R'=2\lambda^{\ell+1}\frac{r_0}{s_0}L\cdot r.
\end{displaymath}
Hence $R=c(\varepsilon)\cdot r$ with $c(\varepsilon)=2\lambda^{\ell+1}\frac{r_0}{s_0}L$, and $\ell=\ell(\varepsilon)$ is independent of $r$. This shows linear wiggle--repetitivity.
\end{proof}

\section{Almost Repetitivity and Minimality}  \label{sec:alm-rep}

In this section $M$ is a non-empty, locally compact and second-countable topological space. We stick to the convention that every locally compact space enjoys the Hausdorff property. We consider a metrisable topological group $T$ acting on $M$ from the left. We assume that the action $(x,m)\mapsto xm$ from $T\times M$ to $M$  is continuous and proper. We fix a $T$-invariant proper metric $d$ on $M$ that generates the topology on $M$. 

\begin{remark}
An action is proper if the map $(x,m)\mapsto (xm,m)$ from $T\times M$ to $M\times M$ is proper, i.e., if pre--images of compact sets in $M\times M$ are compact in $T\times M$, where we use the product topology on $T\times M$ and on $M\times M$.
A metric $d$ on $M$ is proper if all closed balls in $M$ of finite radius are compact.  A $T$-invariant proper metric generating the topology on $M$ indeed exists under our assumptions on the group action, see \cite[Thm~4.2]{amn} and, for a detailed discussion in our context, \cite[Section~2.1]{MR}.
\end{remark}

\begin{example}
Our prime example is $M=\mathbb R^\mathsf{d}$ with the Euclidean metric $d$, together with the canonical left action on $M$ of $T=\mathbb R^\mathsf{d}$ or $T=E(\mathsf{d})$, the Euclidean group $E(\mathsf{d})=\mathbb R^\mathsf{d}\rtimes O(\mathsf{d})$ with the standard topology. For $(x,r)\in E(\mathsf{d})$ and $m\in \mathbb R^\mathsf{d}$, with  $x$ a translation and $r$ a rotation or rotation--reflection about the origin, we write $(x,r)m=x+r\cdot m$.
\end{example}

Let $B_s(m)$ denote the open ball in $M$ of radius $s>0$ centered in $m$.
A subset $P$ of $M$ is called \emph{uniformly discrete of radius $r>0$}, if every open ball in $M$ of radius $r$ contains at most one point of $P$. It is called \emph{relatively dense of radius $R>0$}, if every closed ball in $M$ of radius $R$ contains at least one point of $P$. A \emph{Delone set} is a subset of $M$ which is uniformly discrete and relatively dense. Let the space $\mathcal{P}_r$ of uniformly discrete sets in $(M,d)$ of radius $r>0$ be equipped with the local rubber metric $d_{LR}$. Then $\mathcal{P}_r$ is compact by standard reasoning, compare \cite[Remarks~2.10]{MR}. We also call the elements of $\mathcal{P}_r$ \emph{point sets}. 

For a uniformly discrete set $P\subseteq M$ and $V\subseteq M$ bounded the product set $P\sqcap V:=(P\cap V)\times V$ is called a {\em pattern (of $P$)}, $V$ is called the \emph{support} of the pattern, and the finite set $P\cap V$ is called the \emph{content} of the pattern. An \emph{$s$-pattern} 
is a pattern $P\sqcap V$, where $V$ is a closed ball of finite radius $s>0$.  Every $s$-pattern is also called a \emph{ball pattern}.

Repetitivity describes how equivalent patterns repeat in a point set, where equivalence is understood with respect to the group $T$.
We call two subsets $N,N'$ of $M$ \emph{equivalent (w.r.t.~$T$)},  if there exists $t\in T$ such that $tN=N'$, and we call $N'$ a \emph{shift} of $N$. 
We call two patterns $P\sqcap V$ and $P'\sqcap V'$ \emph{equivalent (w.r.t.~$T$)}, if there exists $t\in T$ such that $tP\sqcap tV=P'\sqcap V'$.

In order to measure the deviation of two point sets $P,P'\in \mathcal{P}_r$ within bounded $V\subseteq M$, we use the distance
\begin{displaymath}
d_V(P,P'):=\inf\{\varepsilon>0\,|\,P\cap V\subseteq(P')_\varepsilon\text{~and~} P'\cap V\subseteq(P)_\varepsilon \},
\end{displaymath}
where $(P)_\varepsilon=\cup_{p\in P}B_\varepsilon(p)$.
If $P,P'$ have no points outside $V$, then $d_V(\cdot,\cdot)$ is the Hausdorff distance between $P$ and $P'$. In general, a small distance $d_V(P,P')$ is compatible with large deviations between $P$ and $P'$ near the boundary of $V$, in contrast to the Hausdorff distance between $P\cap V$ and $P'\cap V$. 

Given $\varepsilon\ge0$, we say that two patterns $P\sqcap V$ and $P'\sqcap V'$ are \emph{$\varepsilon$--similar}, $(P\sqcap V)\sim_\varepsilon(P'\sqcap V')$, if there exists $t\in T$ such that $V'=tV$ and $d_{tV}(tP,P')\le\varepsilon$.
This relation is reflexive. It is also symmetric, due to $T$--invariance of the metric on $M$. It is not transitive in general. If $(P\sqcap V)\sim_{\delta}(P'\sqcap V')$ and $(P'\sqcap V')\sim_{\varepsilon}(P''\sqcap V'')$, then $(P\sqcap V)\sim_{\delta+\varepsilon}(P''\sqcap V'')$, due to the triangle inequality.

\begin{lemma}\label{epsFLC}
Let $P$ be a uniformly discrete set and let $V\subseteq M$ be compact. Take arbitrary $\varepsilon>0$. Then the collection of patterns of $P$ supported on shifts of $V$ can be subdivided into finitely many classes of $\varepsilon$--similar patterns. 
\end{lemma}

\begin{remark}
Due to the above lemma, there is no need for an ``almost version'' of finite local complexity, see also Lemma~\ref{lem:ar}.
\end{remark}

\begin{proof}
Take arbitrary $\varepsilon>0$ and fix some finite cover $(B_{\varepsilon/2}(m_i))_{i \in I}$ of $V$. For a pattern $P\sqcap V'$ such that $V'=tV$ for some $t\in T$, define $J=J(V')\subseteq I$ by
\begin{displaymath}
J:=\{j\in I\,|\,t^{-1}(P\cap V')\cap B_{\varepsilon/2}(m_j)\ne\varnothing\}.
\end{displaymath}
We say that the pattern $P\sqcap V'$ is \emph{of type $J$}. It is easy to see  that patterns of $P$, which are supported on shifts of $V$ and are of the same type, are in fact $\varepsilon$--similar. But the number of different types $J\subseteq I$ is finite, since $I$ is finite. Hence the subdivision into types leads to classes of $\varepsilon$--similar patterns.
\end{proof}

For the following definition we call $L\subseteq T$ \textit{relatively dense} if there exists compact $K=K(L)\subseteq T$ such that $KL=T$. 

\begin{definition}\label{defawr}
$P\in\mathcal{P}_r$ is called \emph{almost repetitive}, if for every $\varepsilon>0$ and for every compact $V\subseteq M$ the set
\begin{displaymath}
T_{V,\varepsilon}(P):=\{x\in T\,|\,  d_V(xP,P)<\varepsilon\}
\end{displaymath}
is relatively dense in $T$.
\end{definition}

Almost repetitivity generalises (weak) repetitivity of the previous section, as is seen from the following lemma.

\begin{lemma}\label{lem:ar}
Assume that the group action on $M$ is also transitive. Then for $P\in\mathcal{P}_r$ the following statements are equivalent.
\begin{itemize}
\item[i)] $P$ is almost repetitive.
\item[ii)] For every ball pattern $P\sqcap V$ and for every $\varepsilon>0$ there exists $R=R(V,\varepsilon)>0$, such that every $R$--pattern in $P$ contains a pattern $\varepsilon$--similar to $P\sqcap V$.
\item[iii)] For every $r>0$ and for every $\varepsilon>0$ there exists $R=R(r,\varepsilon)>0$, such that every $R$--pattern in $P$ contains an $\varepsilon$--similar copy of every $r$--pattern in $P$.
\end{itemize}
\end{lemma}

\begin{remark}
(i) Due to ii), an almost repetitive point set is a Delone set, if the group action is also transitive. 

(ii) Examples of almost repetitive point sets arise from weakly repetitive tilings via structure--preserving prototile decorations, compare Section~\ref{sec:dpst}. For further examples, see below.

(iii) A property similar to ii) above is called repetitive in \cite[Def~2.1.6.]{BBG}.
\end{remark}

\begin{example}\label{ex:arep}
For $k\in\mathbb Z\setminus\{0\}$, let $t(k)\ge0$ denote the largest integer such that $2^{t(k)}$ divides $k$, and set $t(0):=\infty$. 
Define $p_k:=k+2^{-(t(k)+1)}$ and note that $|p_{k+1}-p_{k}|\ge 1/2$ for all $k\in\mathbb Z$. We study the uniformly discrete set
\begin{displaymath}
P:=\{p_k\,|\, k\in\mathbb Z\}\subset\mathbb R.
\end{displaymath}
Since $2|p_{2k-1}-p_{2k}|=1+2^{-t(2k)}$ for all $k\in\mathbb Z$, there are infinitely many different distances between consecutive points. Hence $P$ has not FLC and cannot be repetitive, compare Remark~\ref{rem:rep}. In fact $P$ is not weakly repetitive, since a distance of $1/2$ between neighbouring points occurs for $p_{-1}$ and $p_{0}$ only.
To see that $P$ is almost repetitive, we use the estimate
\begin{displaymath}
|p_k+m2^n-p_{k+m2^n}|\le 2^{-(n+1)} \qquad (k\in\mathbb Z, m\in 2\mathbb Z+1, n\in\mathbb N),
\end{displaymath}
which is obtained from estimates of $t(k+m2^n)$ for $n>t(k)$, $n=t(k)$ and $n<t(k)$. We infer
\begin{displaymath}
P+m2^n\subseteq (P)_{2^{-(n+1)}},\qquad P\subseteq (P+m2^n)_{2^{-(n+1)}} \qquad (m\in 2\mathbb Z+1, n\in\mathbb N),
\end{displaymath}
from which almost repetitivity can be read off.  The above inclusions also show that $R(r,\varepsilon)$ in Lemma~\ref{lem:ar} iii) can be chosen linearly in $r$.

These arguments also apply to the uniformly discrete set $P_0:=\{p_k\,|\, k\in\mathbb Z\setminus\{0\}\}\subset\mathbb R$, which has not FLC, but is almost linearly repetitive. In fact $P_0$ is weakly repetitive, which follows from
$p_{k}+m2^n=p_{k+m2^n}$ for $k\in\mathbb Z\setminus\{0\}$, $m\in\mathbb Z$ and $n>t(k)$. This also shows that $P_0$ is  linearly weakly repetitive.
\end{example}

\begin{proof}[Proof of Lemma~\ref{lem:ar}]

i) $\Rightarrow$ ii). 
Take a ball pattern $P\sqcap V$ and $\varepsilon>0$. Then $T_{V,\varepsilon}(P)$ as in Definition \ref{defawr} is relatively dense in $T$. Hence there exists compact $K\subseteq T$ such that $K T_{V,\varepsilon}(P)=T$. Fix $m_o\in M$ and choose $R=R(V,\varepsilon)$ sufficiently large such that $KV\subseteq B_R(m_o)$. Take arbitrary $m\in M$. Then $m_o=kxm$ for some $k\in K$ and some $x\in T_{V,\varepsilon}(P)$ due to transitivity. But then we have $x^{-1}V\subseteq B_R(m)$, since $kV\subseteq B_R(m_o)$ implies that $V\subseteq k^{-1}B_R(m_o)=B_R(xx^{-1}k^{-1}m_o)=B_R(xm)=xB_R(m)$, due to $T$--invariance of the metric. We also have $d_{x^{-1}V}(x^{-1}P,P)=d_V(P,xP)<\varepsilon$, since $x\in T_{V,\varepsilon}(P)$. Hence the pattern $P\sqcap x^{-1}V$ has the property claimed in ii).

ii) $\Rightarrow$ iii).
Fix $r>0$ and $\varepsilon>0$. Take a finite collection $\{P\sqcap V_1,\ldots, P\sqcap V_k\}$ of $r$--patterns, such that every $r$--pattern in $P$ is $(\varepsilon/2)$--similar to some pattern in the collection. Such a collection exists by Lemma \ref{epsFLC}. Define $R=R(r,\varepsilon):=\max\{R(V_1,\varepsilon/2),\ldots,R(V_k,\varepsilon/2)\}$. Now take an arbitrary $R$-pattern and an arbitrary $r$--pattern $P\sqcap V$. We have $(P\sqcap V)\sim_{\varepsilon/2}(P\sqcap V_i)$ for some $i\in\{1,\ldots,k\}$. By assumption, there is some pattern $(P\sqcap V')\sim_{\varepsilon/2}(P\sqcap V_i)$ contained in the $R$-pattern. But this means that $(P\sqcap V')\sim_\varepsilon(P\sqcap V)$, which proves the implication.

iii) $\Rightarrow$ i). Take $\varepsilon>0$ and compact $V\subseteq M$. Assume w.l.o.g.~that $V\ne\varnothing$. Take $r>0$ such that $V$ is contained in some $r$-ball $B_r(m_o)$ and set $R=R(r,\varepsilon)$ as in iii). Let $\bigcup_{i\in\mathbb N}B_R(m_i)$ be a countable cover of $M$ by $R$-balls. For every $i\in\mathbb N$ take $x_i\in T$ such that
\begin{displaymath}
x_i^{-1}B_r(m_o)\subseteq B_R(m_i), \qquad d_{x_i^{-1}B_r(m_o)}(x_i^{-1}P,P)<\varepsilon.
\end{displaymath}
The second of the above properties results for every $x_i\in T$ in the estimate
\begin{displaymath}
d_{V}(P,x_iP)=d_{x_i^{-1}V}(x_i^{-1}P,P)<\varepsilon.
\end{displaymath}
We now define $T_{V,\varepsilon}(P):=\{x_i\in T\,|\,i\in\mathbb N\}$. For $m\in V$ fixed, we have $x_i^{-1}m\in B_R(m_i)$ for every $x_i\in T_{V,\varepsilon}(P)$. Hence we have $M=\bigcup_{i\in\mathbb N}B_{2R}(x_i^{-1}m)=\bigcup_{i\in\mathbb N}x_i^{-1}B_{2R}(m)$, where we used $T$--invariance of the metric. Now the set
\begin{displaymath}
K:=\{x\in T \,|\, x^{-1}m\in B_{2R}(m)\}
\end{displaymath}
is compact, due to properness of the group action. In order to show $KT_{V,\varepsilon}(P)=T$, take arbitrary $x\in T$. Then $x^{-1}m\in x_i^{-1}B_{2R}(m)$ for some $i\in\mathbb N$, and we can conclude
\begin{displaymath}
\begin{split}
x&\in\{y\in T\,|\,y^{-1}m\in x_i^{-1}B_{2R}(m)\}
=\{y\in T\,|\,x_iy^{-1}m\in B_{2R}(m)\}\\
&=\{z\in T\,|\,z^{-1}m\in B_{2R}(m)\}x_i \subseteq K T_{V,\varepsilon}.
\end{split}
\end{displaymath}
As the reverse inclusion is trivial, we have shown that $P$ is almost repetitive.
\end{proof}

Let us define two related notions.

\begin{definition}
\begin{itemize}
\item[i)] $P\in \mathcal{P}_r$ is called \textit{almost periodic}, if for every $\varepsilon>0$ the collection of 
\emph{$\varepsilon$--periods} of $P$, i.e., the set
\begin{displaymath}
T_\varepsilon(P):=\{x\in T\,|\,  d_{LR}(xP,P)<\varepsilon\},
\end{displaymath}
is relatively dense in $T$.
\item[ii)] $P,P'\in \mathcal{P}_r$ are called \textit{almost locally indistinguishible}, if
for every $\varepsilon>0$ and for every compact $V\subseteq M$, there exists $x'\in T$ such that $d_V(P, x'P')<\varepsilon$, and there exists $x\in T$ such that $d_V(xP,P')<\varepsilon$.
\end{itemize}

\end{definition}

\begin{example}\label{ex:lap}
Let $f$ be a \emph{Bohr--almost periodic function}, i.e., a real-valued continuous function that can be uniformly approximated by trigonometric polynomials. Then the sequence $(x_n)_{n\in\mathbb Z}$, where $x_n=f(n)$, is \emph{Bohr--almost periodic}, i.e., for every $\varepsilon>0$ there exists $K=K(\varepsilon)$ such that every sequence of $K$ consecutive integers contains $k$ such that $|x_{n+k}-x_n|<\varepsilon$ for all $n\in\mathbb Z$. The value $k$ is called a \emph{Bohr--$\varepsilon$-period} of $x_n$. Note that $(2^{-(t(n)+1)})_{n\in\mathbb Z}$ in Example~\ref{ex:arep} is a Bohr--almost periodic sequence. Every Bohr--almost periodic sequence derives from a Bohr--almost periodic function, see e.g.~\cite[Thm~1.27]{C}.

For a Bohr--almost periodic function $f$ of norm $||f||_\infty=1/3$, define $P:=\{n+f(n)\,|\,n\in\mathbb Z\}$. Then $P$ is uniformly discrete of radius $r=1/3$, and $P$ is almost periodic w.r.t.~$T=\mathbb R$. To see the latter, fix $\varepsilon\in]0,1/\sqrt{2}[$ and a Bohr--$\varepsilon$--period $k$ as above. Noting that $k^{-1}P=\{n-k+f(n)\,|\,n\in\mathbb Z\}=\{n+f(n+k)\,|\,n\in\mathbb Z\}$, we infer that 
$k^{-1}P\subseteq (P)_\varepsilon$ and $P\subseteq (k^{-1}P)_\varepsilon$. But this means that $d_{LR}(k^{-1}P,P)<\varepsilon$, hence $k^{-1}\in T_\varepsilon(P)$, and relative denseness of $T_\varepsilon(P)$ follows from relative denseness of the Bohr--$\varepsilon$-periods of $(f(n))_{n\in\mathbb N}$. Moreover, for every compact $V$ we have $d_V(k^{-1}P,P)<\varepsilon$, hence $k^{-1}\in T_{V,\varepsilon}(P)$. Hence $P$ is almost repetitive. The argument also shows that we can choose $R(r,\varepsilon)$ in Lemma~\ref{lem:ar} iii) linearly in $r$. 

This  construction also works in higher dimensions and, more generally, for lattices in a locally compact metrisable Abelian group.

\end{example}

For $P\in\mathcal{P}_r$ we consider its translation orbit closure $X_P$ with respect to the local rubber topology, compare Section~\ref{sec1}.
The topological dynamical system $(X_P,T)$ is \textit{minimal} if for any element its $T$--orbit is dense in $X_P$. The following characterisation of minimality is a version of Gottschalk's theorem \cite{G}.

\begin{theorem}\label{charmin}
For $P\in \mathcal{P}_r$ the following are equivalent.
\begin{itemize}
\item[i)] $(X_P,T)$ is minimal.
\item[ii)] For every $P'\in X_P$, the point sets $P,P'$ are almost locally indistinguishible.
\item[iii)] $P$ is almost periodic.
\item[iv)] $P$ is almost repetitive.
\end{itemize}
\end{theorem}

\begin{remark}
We say that a uniformly discrete set $P\subset M$ has FLC w.r.t.~$T$ if, for every $r>0$, there are only finitely many classes
of $T$--equivalent contents of $r$--patterns in $P$. In the case $M=\mathbb R^\mathsf{d}$ with the canonical group action, this is equivalent to $P-P$ being discrete. The latter condition is called $P$ being of finite type \cite{LaPl03}, see also \cite{Yok05}. If we restrict to the subspace $\mathcal F_r\subseteq \mathcal{P}_r$ containing all FLC sets, then the above theorem reduces to the known characterisation of minimality in that situation, compare \cite[Thm~3.2]{LaPl03} and \cite[Prop~4.16]{Yok05}. To see this, note that  the local matching topology and the local rubber topology coincide on $\mathcal F_r$ due to FLC. Moreover, by FLC,  almost repetitivity is equivalent to repetitivity, and almost local indistinguishibility is equivalent to local indistinguishibility.
\end{remark}

\begin{proof}
i) $\Rightarrow$ ii). Take arbitrary $P'\in X_P$.  W.l.o.g.~fix $\varepsilon\in]0,1/\sqrt{2}[$. For compact $V\subset M$ choose $\delta\in ]0,\varepsilon[$ sufficiently small such that $V\subseteq B_{1/\delta}$.  By minimality, we have $X_{P'}=X_P$. We may thus take $x'\in T$ such that $d_{LR}(P,x'P')<\delta$. But then
\begin{equation}\label{itoii}
x'P'\cap V\subseteq x'P'\cap B_{1/\delta}\subseteq (P)_\delta\subseteq (P)_\varepsilon.
\end{equation}
Now we can conclude $d_V(x'P',P)<\varepsilon$, since for the other inclusion
a statement analogous to \eqref{itoii} holds. The remaining estimate is shown similarly.

ii) $\Rightarrow$ iii). W.l.o.g.~let $\varepsilon\in]0,1/\sqrt{2}[$ be given. 
We show that $T_\varepsilon(P)$ 
is relatively dense in $T$. Take arbitrary $P'\in X_P$. Then for any  compact $V \supseteq B_{1/\varepsilon}$ there exists some 
$x'\in T$ such that $d_V(P, x'P')<\varepsilon$, since $P,P'$ are almost locally indistinguishible. This implies
\begin{equation}\label{iitoiii}
P\cap B_{1/\varepsilon}\subseteq P\cap V\subseteq (x'P')_\varepsilon.
\end{equation}
We conclude $d_{LR}(P,x'P')<\varepsilon$, since a statement analogous to \eqref{iitoiii} holds for the other inclusion. Hence every $P'\in X_P$ has a $T$--orbit which meets $\mathcal U_\varepsilon(P):=\{P'\in X_P\,|\,d_{LR}(P',P)<\varepsilon\}$. This means that 
\begin{displaymath}
X_P\subseteq \bigcup_{x\in T}x\,\mathcal U_\varepsilon(P).
\end{displaymath}
Since $X_P$ is compact, there are $x_1,\ldots, x_k\in T$ such that $X_P\subseteq \bigcup_{i=1}^kx_i\,\mathcal U_\varepsilon(P)$. With $K:=\{x_1,\ldots, x_k\}$ compact, we then have $KT_\varepsilon(P)=T$, which shows iii). Indeed, take arbitrary $x\in T$. Then there is $i\in \{1,\ldots,k\}$ such that $xP\in x_i\mathcal U_\varepsilon(P)$,
hence $x_i^{-1}xP\in \mathcal U_\varepsilon(P)$ and $x_i^{-1}x\in T_\varepsilon(P)$. This means that
$x\in x_iT_\varepsilon(P)\subseteq K T_\varepsilon(P)$, which shows that $T\subseteq K T_\varepsilon(P)$. The reverse inclusion is obvious.

iii) $\Rightarrow$ iv). W.l.o.g.~take $\varepsilon\in]0,1/\sqrt{2}[$ and compact $V\subset M$. Choose $\delta\in ]0,\varepsilon[$ sufficiently small such that $V\subseteq B_{1/\delta}$. Now take arbitrary $x\in T_\delta(P)$. We then have
\begin{displaymath}
xP\cap V \subseteq xP\cap B_{1/\delta}\subseteq (P)_\delta\subseteq (P)_\varepsilon.
\end{displaymath}
We conclude that $d_V(xP, P)<\varepsilon$, since for the other inclusion
an analogous statement holds. We thus have shown that $T_\delta(P)\subseteq T_{V,\varepsilon}(P)$, which implies iv).

iv) $\Rightarrow$ i). If $X_P$ was not minimal, there is $P'\in X_P$ such that $X_P\neq X_{P'}$. Using $T$-invariance of the local rubber metric, it can be seen that 
this implies $P\notin X_{P'}$. But then there exists $\varepsilon>0$ such that $\mathcal V:=\overline{B}_\varepsilon(P)$, the closed ball in $X_P$ of radius $\varepsilon$ about $P$, satisfies $\mathcal V\cap X_{P'}=\varnothing$. Take compact $V\supseteq B_{1/\varepsilon}$ and $x\in T_{\varepsilon,V}(P)$. Then $d_V(xP, P)<\varepsilon$. By an argument as in the proof of ii) $\Rightarrow$ iii), we can conclude that $d_{LR}(xP,P)<\varepsilon$, which implies that $xP\in \mathcal V$. Since $P$ is almost repetitive, there exists compact $K_V\subseteq T$ such that $K_VT_{\varepsilon,V}(P)=T$. Hence
\begin{displaymath}
TP=K_VT_{V,\varepsilon}(P)P\subseteq K_V\mathcal V.
\end{displaymath}
Due compactness of $K_V$, compactness of $\mathcal V$ in the compact space $\mathcal{P}_r$, and due to continuity of the group action, we conclude that $X_P\subseteq K_V\mathcal V$. Hence $P'=x\widetilde P$
for some $x\in K_V$ and some $\widetilde P\in\mathcal V$. But this leads to the contradiction $\mathcal V\cap X_{P'}\ne\varnothing$.
\end{proof}

Guided by Lemma~\ref{lem:ar}, Example~\ref{ex:lap} and by the properties of primitive substitution tilings, we are led to consider

\begin{definition}\label{def:alr}
Assume that $T$ acts on $M$ also transitively. Then $P\in\mathcal{P}_r$ is called \emph{almost linearly repetitive
  (w.r.t.~$T$)}, if $P$ is almost repetitive w.r.t.~$T$, and if one can choose $R(r,\varepsilon)=\mathcal O(r)$ as $r\to\infty$ in Lemma~\ref{lem:ar} iii), where the $\mathcal O$--constant may depend on $\varepsilon$.
\end{definition}

In the following two sections, we will study the implications of almost linear repetitivity and relaxed versions thereof. We will restrict to 
$M=\mathbb R^\mathsf{d}$, with $\mathsf{d}\in\mathbb N$, since our approach crucially relies on box decompositions \cite{LaPl03, DaLe01}.

\section{Almost linear repetitivity and unique ergodicity} 

For measurable sets $(A_i)_{i\in I}$ in $\mathbb R^\mathsf{d}$ of positive volume, we call $(A_i)_{i\in I}$ a \emph{decomposition of $A\subseteq \mathbb R^\mathsf{d}$} if $\bigcup_{i\in I} A_i=A$ and if the intersections $A_i\cap A_j$ have measure $0$ for $i\ne j$.  A decomposition $(A_i)_{i\in I}$ is called \emph{box decomposition} if all $A_i$ are boxes, where
a \emph{box} $B\subset \mathbb R^\mathsf{d}$ is a compact set $B=\bigtimes_{i=\mathsf{1}}^\mathsf{d} [a_i,b_i]$ of positive volume. We call  $b_i-a_i>0$ a \emph{side length} and $\omega(B):=\min\{b_i-a_i \, |\, i\in\{\mathsf{1},\ldots,\mathsf{d}\}\}>0$ the \emph{width of $B$}. 
 Let $\mathcal{B}(U)$ denote the collection of \emph{squarish boxes}, i.e., the collection of boxes all of whose side lengths lie in $[U,2U]$.  For $W\ge U$ any squarish box 
in $\mathcal{B}(W)$ has a finite box decomposition into squarish boxes in $\mathcal{B}(U)$, see \cite{LaPl03}. 
Let $\mathcal B:=\bigcup_{U>0} \mathcal B(U)$ denote the collection of all squarish boxes.

\begin{definition}
Let $\mathcal{P}\subseteq \mathcal{P}_r$ be given.   Consider 
a function $w:\mathcal B \times \mathcal{P}\to\mathbb R$ and assume that $w$ satisfies, with 
$w_P(B):=w(B,P)$,
\begin{itemize}
\item[i)] (\textit{boundedness})
For every $P\in \mathcal{P}$ there exists $C_P\ge0$ such that for every $B\in\mathcal B$ we have 
\begin{displaymath}
|w_P(B)|\le C_P \mathrm{vol}(B).
\end{displaymath}
\item[ii)] (\textit{almost subadditivity}) For every $\varepsilon>0$ there exists $U_1>0$ such that for all $U\ge U_1$ the following holds:
If $(B_i)_{i\in I}$ is a box decomposition of $B\in \mathcal B$ satisfying $B_i\in\mathcal B(U)$ for all $i$, then 
for every $P\in\mathcal{P}$ we have
\begin{displaymath}
w_P(B)-\sum_{i\in I} w_P(B_i)\le \varepsilon\,\mathrm{vol}(B).
\end{displaymath}
\item[iii)] (\textit{almost covariance}) For every $\varepsilon>0$ there exists $U_2>0$ such that for all $U\ge U_2$ the following holds:
For every $B\in\mathcal B(U)$,  $P\in\mathcal{P}$ and $x\in \mathbb R^\mathsf{d}$ we have
\begin{displaymath}
|w_{P}(B)-w_{xP}(xB)|\le\varepsilon\, \mathrm{vol}(B).
\end{displaymath}
\item[iv)] 
(\textit{almost invariance}) For every $\varepsilon>0$ there exist $U_3>0$  and $\delta>0$ such that
for all $U\ge U_3$ the following holds: If $B\in\mathcal B(U)$ and $P,P'\in\mathcal{P}$ are given such that $d_{B}(P, P')<\delta$, then
\begin{displaymath} 
|w_{P}(B)- w_{P'}(B)|\le\varepsilon\, \mathrm{vol}(B).
\end{displaymath}
\end{itemize}
Then $w$ is called a \textit{weight function} on $\mathcal B\times\mathcal{P}$ with respect to $d_B$.
\end{definition}

We give an important example of a weight function.

\begin{lemma}\label{weight-function}
For $\widehat P\in \mathcal{P}_r$ consider $\mathcal{P}:=X_{\widehat P}$. Then for every $f\in C(\mathcal{P})$ the function
\begin{displaymath}
(B,P)\mapsto w_P(B):=\int_{B}\mathrm{d}x\,f(x^{-1} P)
\end{displaymath}
is a weight function on $\mathcal B\times \mathcal{P}$ w.r.t.~$d_B$. We may take $C_P=||f||_\infty$ and, for given $\varepsilon>0$, a constant $\delta=\delta(\varepsilon)>0$ of uniform continuity of $f$.
\end{lemma}

In the following proof, we use the \emph{van Hove boundary} $\partial^K B$ of $B\subseteq M$ with respect to $K\subseteq T$. It is defined as $\partial^K B :=[KB\cap B^c]\cup [KB^c\cap B]$, where $KB=\{x m\,|\,x\in K, m\in B\}$, compare \cite[Sec.~2.2]{MR}. It is called boundary since $x\in B\setminus \partial^KB$ implies $K^{-1}x\subseteq B$. It is not difficult so show that for $B_{U}\in \mathcal B(U)$ we have  $\mathrm{vol}(\partial^K B_{U})/\mathrm{vol}(B_{U})\to0$ as $U\to \infty$ for any compact set $K\subset\mathbb R^\mathsf{d}$. We say that squarish boxes satisfy the van Hove property.

\begin{proof}

i). Boundedness holds with $C_P:=||f||_\infty\ge0$ independently of $P\in \mathcal{P}$, which is seen from a standard estimate. 

ii). We show additivity. Consider a finite box decomposition $(B_i)_{i\in I}$ of $B$. 
By the decomposition property we conclude $w_P(B)=\sum_{i\in I} w_P(B_i)$ for every $P\in \mathcal{P}$.

iii). 
Covariance $w(xB,xP)=w(B,P)$ follows from left invariance of the Lebesgue measure.

iv).
Fix $f\in C(\mathcal{P})$.  We may assume w.l.o.g.~that $f$ is not identically vanishing. Let 
$\varepsilon>0$ be given. Due to compactness of $\mathcal{P}$, the function $f\in C(\mathcal{P})$ is 
uniformly continuous on $\mathcal{P}$. Hence we may take $\delta\in]0,1/\sqrt{2}[$ such that for all $P,P'\in
\mathcal{P}$ satisfying $d_{LR}(P,P')<\delta$ we have $|f(P)-f(P')|<\varepsilon$. 
For $B_U\in \mathcal{B}(U)$, define $\widetilde B_U:=B_U\setminus\partial^{B_{1/\delta}}B_U$. 
Due to the van Hove property of squarish boxes, we may choose $U_3>0$ such that for all $U\ge U_3$ and for all 
$B_U\in\mathcal B(U)$ we have
\begin{displaymath}
\mathrm{vol}(\partial^{B_{1/\delta}}B_U)\le \frac{\varepsilon}{||f||_\infty}\mathrm{vol}(B_U).
\end{displaymath}
This results for every $U\ge U_3$, for every $B_U\in\mathcal B(U)$ and for every 
$P\in \mathcal{P}$ in the estimate
\begin{equation}\label{tr1}
\left|\int_{B_U} f(x^{-1}P)\,{\rm d}x-\int_{\widetilde B_U} f(x^{-1}P)\,{\rm d}x\right|
=\left|\int_{\partial^{B_{1/\delta}}B_U} f(x^{-1}P)\,{\rm d}x\right|
\le \varepsilon\,\mathrm{vol}(B_U).
\end{equation}
Now fix arbitrary $U\ge U_3$, and let $P,P'\in\mathcal{P}$ and $B_U\in\mathcal B(U)$ be given such that
$d_{B_U}(P, P')<\delta$. We then have for all $x\in \widetilde B_U$ that 
$d_{LR}(x^{-1}P,x^{-1}P')<\delta$.
To see this, note that $xB_{1/\delta}\subseteq B_U$ by the remark before the proof of the lemma, due to inversion 
invariance of $B_{1/\delta}$ and commutativity of the group $(\mathbb R^\mathsf{d},+)$. Hence we can 
estimate
\begin{displaymath}
P'\cap xB_{1/\delta}\subseteq P'\cap B_U\subseteq (P)_\delta,
\end{displaymath}
which implies $x^{-1}P'\cap B_{1/\delta}\subseteq (x^{-1}P)_\delta$ by translation invariance of the metric. The other inclusion is shown analogously. This results in the estimate
\begin{equation}\label{tr2}
\left|\int_{\widetilde B_U} f(x^{-1}P')\,{\rm d}x-\int_{\widetilde B_U} f(x^{-1}P)\,{\rm d}x\right|\le 
\varepsilon\,\mathrm{vol}(\widetilde B_U)\le \varepsilon\,\mathrm{vol}(B_U).
\end{equation}
We can now use \eqref{tr1} and \eqref{tr2} and a $3\varepsilon$-argument to obtain
\begin{equation}\label{tr1main}
\left|\int_{B_U} f(x^{-1}P')\,{\rm d}x-\int_{B_U} f(x^{-1}P)\,{\rm d}x\right|\le 3\varepsilon\,\mathrm{vol}(B_U).
\end{equation}
As $\varepsilon>0$ was arbitrary, this shows almost invariance.
\end{proof}

For $U>0$, we define upper and lower local densities
\begin{displaymath}
f_P^+(U):=\sup\left\{\frac{w_P(B)}{\mathrm{vol}(B)}\,|\, B\in\mathcal B(U)\right\},\quad
f_P^-(U):=\inf\left\{\frac{w_P(B)}{\mathrm{vol}(B)}\,|\, B\in\mathcal B(U)\right\}.
\end{displaymath}
Due to boundedness, these are finite numbers. As a preparation of the following proposition, we consider a variant of sequence monotonicity. We call a sequence $(a_n)_{n\in\mathbb N}$ of real numbers \emph{almost monotonically decreasing}, if for every $\varepsilon>0$ there exists $n_0=n_0(\varepsilon)$ such that for every $n\ge n_0$ there exists $m_0=m_0(n,\varepsilon)$ such that for every $m\ge m_0$ we have $a_m\le a_n+\varepsilon$.

\begin{lemma}\label{almostmono}
Every almost monotonically decreasing sequence of real numbers which is bounded from below converges to its infimum.
\end{lemma}

\begin{proof}
Let $(a_n)_{n\in\mathbb N}$ be any almost monotonically decreasing sequence which is bounded from below. Its infimum $a:=\inf \{a_n\,|\,n\in\mathbb N\}$ is finite, since $(a_n)_{n\in\mathbb N}$ is bounded from below. Consider arbitrary $\varepsilon>0$ and choose $n_0=n_0(\varepsilon)$ as in the definition above. Since $a$ is the infimum of $(a_n)_{n\in\mathbb N}$, we can choose $n_1\ge n_0$ such that $a_{n_1}\le a+\varepsilon$.
Choose $m_0=m_0(n_1,\varepsilon)$ as in the definition above. We then have $a_m\le a_{n_1}+\varepsilon$ for all $m\ge m_0$. On the other hand $a\le a_m$ for all $m$ by definition. Hence
$a\le a_m\le a+2\varepsilon$ for all $m\ge m_0$. Since $\varepsilon>0$ was arbitrary, this means that $(a_n)_{n\in\mathbb N}$ converges to $a$.
\end{proof}

\begin{proposition}\label{mainlemma}
For $\widehat P\in \mathcal{P}_r$, let  $w$ be a weight function on  $\mathcal B\times \mathcal{P}$ w.r.t.~$d_B$, where $\mathcal{P}=X_{\widehat P}$. Then the
following hold.

\begin{itemize}

\item[i)]
For every $P\in\mathcal{P}$, the density $f_{P}^+(U)$ is almost monotonically decreasing in $U$ and converges to a finite limit $f_P$ as $U\to\infty$. 

\item[ii)] Let $\widehat P$ be almost repetitive w.r.t.~$\mathbb R^\mathsf{d}$. Then the limit $f_P$ in i) is independent of the choice of $P\in\mathcal{P}$. In addition, the asymptotic behaviour of the density  $f_P^-(U)$ is independent of the choice of $P\in \mathcal{P}$, i.e., for every $P\in \mathcal{P}$ and for every $\varepsilon>0$ there exists $U_1>0$ such that for all $U\ge U_1$ we have
\begin{displaymath}
\left|f_P^-(U)-f_{\widehat P}^-(U)\right|\le \varepsilon.
\end{displaymath}
\item[iii)] Let $\widehat P$ be almost linearly repetitive w.r.t.~$\mathbb R^\mathsf{d}$. Then for all $P\in \mathcal{P}$ the densities $f_{P}^+(U)$ and $f_{P}^-(U)$ converge to the same limit $f$ as $U\to\infty$. For every sequence $(B_n)_{n\in\mathbb N}$ of squarish boxes such that $\omega(B_n)\to\infty$ as $n\to\infty$, we then have
\begin{displaymath}
\lim_{n\to\infty} \frac{w_{P}(B_n)}{\mathrm{vol}(B_n)}=f,
\end{displaymath}
and this convergence is uniform in the center of the boxes.
\end{itemize}
\end{proposition}

In the proof of Proposition~\ref{mainlemma} iii), we will use the following obvious characterisation of almost linear repetivity, which is adapted to our setup with squarish boxes.
A point set $P\in \mathcal{P}_r$ is almost linearly repetitive w.r.t.~$\mathbb R^\mathsf{d}$, iff for every $\varepsilon>0$ there exists $K=K_\varepsilon>0$ and $U_0=U_0(\varepsilon)>0$ such that for every $U\ge U_0$, for every $B\in\mathcal{B}(KU)$ and for every $B_U\in\mathcal B(U)$, there is $y\in \mathbb R^\mathsf{d}$ such that $yB_U\subseteq B$ and $d_{yB_U}(yP, P)<\varepsilon$.

\begin{proof}

i). For fixed $P\in \mathcal{P}$, choose arbitary $\varepsilon>0$ and $U\ge U_1$ in the definition of the weight function.
Take $W\ge U$, choose arbitrary $B\in\mathcal{B}(W)$
and a box decomposition $(B_i)_{i \in I}$ of $B$ with boxes $B_i\in\mathcal{B}(U)$. 
Using almost subadditivity, we get
\begin{displaymath}
\frac{w_P(B)}{\mathrm{vol}(B)}\le\sum_{i\in I} \frac{w_P(B_i)}{\mathrm{vol}(B_i)}
\frac{\mathrm{vol}(B_i)}{\mathrm{vol}(B)}+\varepsilon\le f_P^+(U)+\varepsilon.
\end{displaymath}
 Since $B\in\mathcal{B}(W)$ was arbitrary, we infer $f_P^+(W)\le f_P^+(U)+\varepsilon$. Hence $f_P^+(U)$ is almost monotonically decreasing in $U$.
Now boundedness implies that $f_P^+(U)$ converges to a finite limit $f_P$ as $U\to\infty$, due to Lemma~\ref{almostmono}.

ii). Fix $P\in\mathcal{P}$ and $\varepsilon>0$. Choose $U_2>0$, $U_3>0$ and $\delta>0$ as in the definition of the weight function. 
For $U\ge U_1:=\max\{U_2,U_3\}$, take arbitrary $B\in\mathcal{B}(U)$. By definition 
of $\mathcal{P}$, there exists $x=x(P,B)\in\mathbb R^\mathsf{d}$ such that $d_B(x\widehat P, P)<\delta$. Hence  almost covariance and almost invariance yield the estimate
\begin{displaymath}
\Big| \frac{w(B,P)}{\mathrm{vol}(B)}-\frac{w(x^{-1}B,\widehat P)}{\mathrm{vol}(B)}\Big|
\le \Big| \frac{w(B,P)}{\mathrm{vol}(B)}-\frac{w(B,x\widehat P)}{\mathrm{vol}(B)}\Big|+\varepsilon
\le 2\varepsilon.
\end{displaymath}
By minimality, compare Theorem~\ref{charmin} ii), an analogous statement holds with $P$ und $\widehat P$ interchanged.
As $B\in \mathcal B(U)$ was arbitrary, we conclude that
\begin{displaymath}
\left|f_{\widehat P}^+(U)-f_P^+(U)\right|\le 2\varepsilon, \qquad \left|f_{\widehat P}^-(U)-f_P^-(U)\right|\le 2\varepsilon.
\end{displaymath}
The claim now follows together with i).

iii). 
We show the statement for $\widehat P$ first. Then the claim follows for all $P\in\mathcal{P}$ from ii).
The inequality $\limsup_{U\to\infty}f_{\widehat P}^-(U)\le \lim_{U\to\infty}f_{\widehat P}^+(U)$ holds trivially.
Assume that $\liminf_{U\to\infty}f_{\widehat P}^-(U)< \lim_{U\to\infty}f_{\widehat P}^+(U)=f$. Then there
are $\varepsilon>0$ and $B_{U_k}\in\mathcal B(U_k)$ for $k\in \mathbb N$, such that
$U_k\to\infty$ and
\begin{equation}\label{form:ass1}
\frac{w_{\widehat P}(B_{U_k})}{\mathrm{vol}(B_{U_k})}\le f-\varepsilon.
\end{equation}
Due to almost covariance we may choose $k_2\in\mathbb N$ such that 
for all $k\ge k_2$ and all $x\in\mathbb R^\mathsf{d}$ we have
\begin{displaymath}
|w_{\widehat P}(B_{U_k})-w_{x\widehat P}(xB_{U_k})|\le\frac{1}{2^{2\mathsf{d}+3}}\,\varepsilon\, \mathrm{vol}(B_{U_k}).
\end{displaymath}
Due to almost invariance we may choose $\delta>0$ and $k_3\in\mathbb N$ such that for every $k\ge k_3$ and for all $P\in\mathcal{P}$ the estimate
\begin{displaymath}
|w_{P}(B_{U_k})-w_{\widehat P}(B_{U_k})|\le \frac{1}{2^{2\mathsf{d}+3}}\,\varepsilon\, \mathrm{vol}(B_{U_k}) 
\end{displaymath}
holds, whenever $d_{B_{U_k}}(P, \widehat P)<\delta$.
Choose a constant $K_\delta$ of almost linear repetitivity and a corresponding $k_0=k_0(\delta)\in\mathbb N$.
Due to almost subadditivity we may choose $k_1\in\mathbb N$ such that for all $k\ge k_1$ and
 for every decomposition of a box $B$ in boxes $B_i\in\mathcal B(U_k)$ we have
\begin{displaymath}
w_P(B)-\sum_{i} w_P(B_i)\le \frac{1}{4}\frac{\varepsilon}{(6K_\delta)^\mathsf{d}}\,\mathrm{vol}(B).
\end{displaymath}
Now fix $k\ge k_4:=\max\{k_0,k_1,k_2,k_3\}$ and take arbitrary $B\in\mathcal{B}(3K_\delta U_k)$.
By partitioning each side of $B$ into $3$ parts of equal length, $B$ can 
be decomposed into $3^\mathsf{d}$
equivalent smaller boxes, each belonging to $\mathcal{B}(K_\delta U_k)$. Denote by
$B^{(i)}\in\mathcal{B}(K_\delta U_k)$ the box which does not touch the topological boundary of $B$. 
By linear almost repetitivity, there exists $y\in \mathbb R^\mathsf{d}$ such that $B_0=yB_{U_k}\subseteq B^{(i)}$ and  $d_{B_0}(y\widehat P, \widehat P)<\delta$. Then almost covariance and almost invariance yield the estimate
\begin{equation}\label{al-inv}
\begin{split}
|w_{\widehat P}(B_0)-&w_{\widehat P}(B_{U_k})|\le
|w_{y^{-1} \widehat P}(B_{U_k})-w_{\widehat P}(B_{U_k})|+\frac{1}{2^{2\mathsf{d}+3}}\,\varepsilon\, \mathrm{vol}(B_{U_k})\\
&\le\frac{1}{2^{2\mathsf{d}+3}}\,\varepsilon\, \mathrm{vol}(B_{U_k})+\frac{1}{2^{2\mathsf{d}+3}}\,\varepsilon\, \mathrm{vol}(B_{U_k})
= \frac{1}{2^{2\mathsf{d}+2}}\,\varepsilon\, \mathrm{vol}(B_0),
\end{split}
\end{equation}
since $d_{B_{U_k}}(y^{-1}\widehat P,\widehat P)=d_{B_0}(\widehat P,y\widehat P)<\delta$.
Using $B\in\mathcal{B}(3K_\delta U_k)$ and $B_0\in\mathcal{B}(U_k)$, we 
may estimate
\begin{equation}\label{eqn:est1}
\frac{1}{(6 K_\delta)^\mathsf{d}}= \frac{U_k^\mathsf{d}}{(2\cdot 3 K_\delta U_k)^\mathsf{d}}\le
 \frac{\mathrm{vol}(B_0)}{\mathrm{vol}(B)}\le \frac{(2\cdot U_k)^\mathsf{d}}{(3 K_\delta U_k)^\mathsf{d}}=
\frac{1}{\left(\frac{3}{2} K_\delta\right)^\mathsf{d}}.
\end{equation}
By construction, we may choose a box decomposition $(B_i)_{i=0}^n$ of $B$, with 
$B_i\in\mathcal{B}(U_k)$ for $i\in\{1,\ldots,n\}$.
Now almost subadditivity, the estimates \eqref{al-inv}, \eqref{form:ass1}, \eqref{eqn:est1} and i) yield
\begin{displaymath}
\begin{split}
\frac{w_{\widehat P}(B)}{\mathrm{vol}(B)}&\le 
\sum_{i=1}^n\frac{w_{\widehat P}(B_i)}{\mathrm{vol}(B)}
+\frac{w_{\widehat P}(B_0)}{\mathrm{vol}(B)}+\frac{1}{4}\frac{\varepsilon}{(6K_\delta)^\mathsf{d}}\\
&\le\sum_{i=1}^n\frac{w_{\widehat P}(B_i)}{\mathrm{vol}(B)}
+\frac{w_{\widehat P}(B_{U_k})}{\mathrm{vol}(B)}+
\frac{1}{2^{2\mathsf{d}+2}} \,\varepsilon\,\frac{\mathrm{vol}(B_0)}{\mathrm{vol}(B)}+\frac{1}{4}\frac{\varepsilon}{(6K_\delta)^\mathsf{d}}\\
&\le
\sum_{i=1}^nf_{\widehat P}^+(U_k)\frac{\mathrm{vol}(B_i)}{\mathrm{vol}(B)}
+\left(f-\varepsilon\right)\frac{\mathrm{vol}(B_0)}{\mathrm{vol}(B)}
+\frac{1}{2}\frac{\varepsilon}{(6 K_\delta)^\mathsf{d}}\\
&\le
f_{\widehat P}^+(U_k)+\left(f-f_{\widehat P}^+(U_k)\right)\frac{\mathrm{vol}(B_0)}{\mathrm{vol}(B)}
-\frac{1}{2}\frac{\varepsilon}{(6 K_\delta)^\mathsf{d}}\\
&\le 
f_{\widehat P}^+(U_k)-\frac{1}{2}\frac{\varepsilon}{(6 K_\delta)^\mathsf{d}}.
\end{split}
\end{displaymath}
Since $B\in\mathcal{B}(3 K_\delta U_k)$ was arbitrary, this implies
$f_{\widehat P}^+(3K_\delta U_k)\le f_{\widehat P}^+(U_k)-\varepsilon/2(6 K_\delta)^\mathsf{d}$. As $k\ge k_4$
was arbitrary, we may take the limit $k\to\infty$ and arrive at a contradiction.

Now let $(B_n)_{n\in\mathbb N}$ be a sequence in $\mathcal B$ such that
$\omega(B_n)\to\infty$ as $n\to\infty$. Since $B_n\in \mathcal B(U_n)$ for some $U_n$, we have the estimate
\begin{displaymath}
f_{\widehat P}^-(U_n)\le\frac{w_{\widehat P}(B_n)}{\mathrm{vol}(B_n)}\le f_{\widehat P}^+(U_n).
\end{displaymath}
Since $U_n\to\infty$ as $n\to\infty$, this yields the claimed uniform convergence. 
\end{proof}

We apply the previous proposition to the above example.

\begin{proposition} \label{thm:linarep-uniqerg}
Let $\widehat P\in\mathcal{P}_r$ be almost linearly repetitive w.r.t.~$\mathbb R^\mathsf{d}$. Then $X_{\widehat P}$ is uniquely ergodic
w.r.t.~$T=\mathbb R^\mathsf{d}$.
\end{proposition}

\begin{remark}
Restricting to the FLC case, we recover the result that linear repetitivity implies unique ergodicity, see \cite[Thm~6.1]{LaPl03}, \cite[Thm~2.7]{LeMo02} and \cite[Cor~4.6]{DaLe01}. 
\end{remark}

\begin{proof}

Unique ergodicity of $X_{\widehat P}$ w.r.t.~$T=\mathbb R^\mathsf{d}$ means that there is exactly one $\mathbb R^\mathsf{d}$-invariant Borel probability measure on $X_{\widehat P}$. Unique ergodicity holds if the volume averages
\begin{displaymath}
J_n(f,P):=\frac{1}{\mathrm{vol}(B_n)}\int_{B_n}\mathrm{d}x f(x^{-1}P),
\end{displaymath}
with $(B_n)_{n\in\mathbb N}$ the sequence of hypercubes in $\mathbb R^\mathsf{d}$ of sidelength $2n$ centered at the origin, converge for all $f\in C(X_{\widehat P})$ and all $P\in X_{\widehat P}$ as $n\to\infty$, with a limit which does not depend on the choice of $P$. This follows from the uniform ergodic theorem, see e.g.~\cite[Thm~2.16]{MR}, by inversion invariance of the Lebesgue measure and inversion invariance of the centered hypercube.
But the latter condition is indeed satisfied, due to Lemma~\ref{weight-function} and Proposition~\ref{mainlemma} iii).
\end{proof}

Unique ergodicity actually holds w.r.t.~$T=\mathbb R^\mathsf{d}\rtimes H$, with $H$ any subgroup of $O(\mathsf{d})$. This can be shown by adapting the previous arguments.

\begin{lemma}\label{lem:wigwf1}
Let $H$ be any subgroup of $O(\mathsf d)$.
For $\widehat P\in \mathcal{P}_r$ consider $\mathcal{P}:=X_{\widehat P}$. Then for every $f\in C(\mathcal{P})$, the function
\begin{displaymath}
(B,P)\mapsto w_P(B):=\int_{B\times H}{\rm d}x f(x^{-1}P)
\end{displaymath}
is a weight function on $\mathcal B\times\mathcal{P}$ w.r.t.~$d_B$.
\end{lemma}

\begin{proof}[Sketch of proof]
Boundedness, additivity and covariance are clear. The proof of almost invariance 
is analogous to that of Lemma~\ref{weight-function}. We describe the modifications. We consider integrals 
over $B_U\times H$ instead of integrals over $B_U$. As van Hove boundary we use
\begin{displaymath}
\partial^{B_{1/\delta}\times\{e\}}\left(B_U\times H\right)=\left(\partial^{B_{1/\delta}} B_U\right)\times H
\end{displaymath}
instead of $\partial^{B_{1/\delta}} B_U$, with $e$ the identity in $O(\mathsf{d})$. This boundary term is of asymptotically small volume $o(\mathrm{vol}(B_U))$ as $U\to\infty$, since the same is true of $\partial^{B_{1/\delta}} B_U$.
A calculation shows that $x\in (B_U\times H)\setminus \partial^{B_{1/\delta}\times\{e\}}(B_U\times H)$ still implies $xB_{1/\delta}\subseteq B_U$, by inversion and rotation invariance of $B_{1/\delta}$ and by the remark before the proof of Lemma~\ref{weight-function}. Almost invariance follows then with $\mathbb R^\mathsf{d}\rtimes H$-invariance of the Euclidean metric.
\end{proof}

\begin{theorem}\label{thm:alruni}
Let $\widehat{P}\in \mathcal{P}_r$ be almost linearly repetitive w.r.t.~$\mathbb R^\mathsf{d}$. Then $X_{\widehat P}$ is uniquely ergodic w.r.t.~$T=\mathbb R^\mathsf{d}\rtimes H$, where $H$ is any subgroup of $O(\mathsf{d})$.
\end{theorem}

\begin{proof}
The proof is analogous to that of Proposition~\ref{thm:linarep-uniqerg}. Let $D_s$ denote the closed ball of radius $s$ centered at the origin in $\mathbb R^\mathsf{d}$. Then a calculation shows that $(D_n\times H)_{n\in\mathbb N}$ is a van Hove sequence in $\mathbb R^\mathsf{d}\rtimes H$, i.e., $\mathrm{vol}(\partial^K(D_n\times H))/\mathrm{vol}(D_n\times H)\to0$ as $n\to\infty$, for every compact $K\subset \mathbb R^\mathsf{d}\rtimes H$. This holds since $(D_n)_{n\in\mathbb N}$ has the van Hove property in $\mathbb R^\mathsf{d}$, and since the centered balls are inversion invariant. Hence we may study the convergence of
\begin{displaymath}
J(f,P,D_n):=\frac{1}{\mathrm{vol}(D_n\times H)}\int_{D_n\times H}\mathrm{d}x f(x^{-1}P)
\end{displaymath}
in order to check unique ergodicity w.r.t.~$T=\mathbb R^\mathsf{d}\rtimes H$. But this can be done by replacing $D_n$ by a squarish box. To see this, note first that
\begin{equation}\label{eq:svH}
\frac{\mathrm{vol}(\partial^{D_{\sqrt{n}}}D_n)}{\mathrm{vol}(D_n)}\longrightarrow 0\qquad (n\to\infty).
\end{equation}
For every $n$, choose a finite decomposition of $D_n$ into squarish boxes $B_i^{(n)}\in\mathcal B(\sqrt{n})$ and a set $B^{(n)}$ which does
not contain any box in $\mathcal B(\sqrt{n})$, such that
\begin{displaymath}
D_n=B^{(n)} \cup \bigcup_{i} B_{i}^{(n)}, \qquad B_i^{(n)}\in\mathcal B(\sqrt{n}), \qquad
B^{(n)}\subseteq \partial^{D_{\sqrt{n}}}D_n.
\end{displaymath}
Now fix arbitrary $\varepsilon>0$. By Proposition~\ref{mainlemma}, Lemma~\ref{lem:wigwf1} and \eqref{eq:svH}, there exist $J(f)\in\mathbb R$ and $n_0=n_0(\varepsilon)$ such that for every $n\ge n_0$, for all $B_i^{(n)}, B^{(n)}$ and for all $P\in X_{\widehat P}$ we have
\begin{displaymath}
\left|J(f,P,B_i^{(n)})-J(f)\right|<\varepsilon, \qquad 
\frac{\mathrm{vol}(B^{(n)})}{\mathrm{vol}(D_n)}<\varepsilon.
\end{displaymath}
Then we have for every $n\ge n_0$ that
\begin{displaymath}
\begin{split}
|J(f,P,&D_n)-J(f)|=|J\big(f,P,B^{(n)}\cup \bigcup_{i} B_{i}^{(n)}\big)-J(f)|\\
&=\left|\frac{1}{\mathrm{vol}(D_n\times H)}\left(w_P(B^{(n)})+\sum_i w_P(B_i^{(n)})\right)-J(f)\right|\\
&\le \left|\sum_i J(f,P,B_i^{(n)})\frac{\mathrm{vol}(B_i^{(n)})}{\mathrm{vol}(D_n)}-J(f)\right|+||f||_\infty\varepsilon\\
&\le \sum_i \left|J(f,P,B_i^{(n)})-J(f)\right|\frac{\mathrm{vol}(B_i^{(n)})}{\mathrm{vol}(D_n)}+(|J(f)|+||f||_\infty)\varepsilon\\
&\le (1+|J(f)|+||f||_\infty)\varepsilon,
\end{split}
\end{displaymath}
where we used the factorisation property $\mathrm{vol}(D_n\times H)=\mathrm{vol}(D_n)\cdot\mathrm{vol}(H)$.
As $\varepsilon>0$ was arbitrary, we conclude $\lim_{n\to\infty} J(f,P,D_n)=J(f)$ for all $P\in X_{\widehat P}$. This proves the claim.
\end{proof}

\section{Almost Linear Wiggle-Repetitivity}  \label{sec:alm-wig-rep}

Motivated by substitution tilings with dense tile orientations, we finally discuss a point set version of linear wiggle--repetitivity. For a ball or a squarish box $B\subset\mathbb R^\mathsf{d}$, we denote by $r_B$ a rotation about  the center of $B$. For some fixed metric on $O(\mathsf{d})$ generating the topology of $O(\mathsf{d})$, the distance of $r_B$ to the identity is denoted by $d(r_B)$. In order to quantify the deviation of two uniformly discrete point sets $P,P'\subset \mathbb R^\mathsf{d}$ within $B$, we consider
\begin{displaymath}
\widetilde d_B(P,P'):=\inf\{\varepsilon>0\,|\,\exists r_B,r_B': \max\{d_B(r_BP,r_B'P'),
d(r_B),d(r_B')\}<\varepsilon \}.
\end{displaymath}
Given $\varepsilon\ge0$, we say that two patterns $P\sqcap B$ and $P'\sqcap B'$ are \emph{$\varepsilon$--wiggle--similar}, if there exists $t\in \mathbb R^\mathsf{d}$ such that $B'=tB$ and $\widetilde d_{tB}(tP,P')\le\varepsilon$.

\begin{definition} 
Let $P$ be a uniformly discrete subset of $\mathbb R^\mathsf{d}$. $P$ is called
\begin{itemize}
\item[i)] \emph{almost wiggle--repetitive},
if for every $r>0$ and for every $\varepsilon>0$ there exists $R=R(r,\varepsilon)>0$, such that every $R$--pattern in $P$ contains an $\varepsilon$--wiggle--similar copy of every $r$--pattern in $P$.
\item[ii)] \emph{almost linearly wiggle--repetitive}, if $P$ is almost wiggle--repetitive, and if one can choose $R(r,\varepsilon)=\mathcal O(r)$ as $r\to\infty$, where the $\mathcal O$--constant may depend on $\varepsilon$. 
\end{itemize}
\end{definition}

\begin{remark}\label{rem:alwr}
It is not hard to see that almost wiggle--repetitivity is equivalent to almost repetitivity w.r.t.~$\mathbb R^\mathsf{d}$. Examples of almost linearly wiggle--repetitive point sets are obtained from linearly wiggle--repetitive tilings of Section~\ref{chap:subs}. The pinwheel tiling point set in Figure~\ref{fig:pin} shows that almost linear wiggle--repetitivity does not imply almost linear repetitivity w.r.t.~$\mathbb R^\mathsf{d}$.  
\end{remark}

We want to show that almost linear wiggle--repetitivity implies unique ergodicity. This can be done by refining the arguments of the previous section.

\begin{lemma}\label{lem:wigwf2}
For $\widehat P\in \mathcal{P}_r$ consider $\mathcal{P}:=X_{\widehat P}$. Then for every $f\in C(\mathcal{P})$ the function
\begin{displaymath}
(B,P)\mapsto w_P(B):=\int_B{\rm d}x f(x^{-1}P)
\end{displaymath}
is a weight function on $\mathcal B\times \mathcal{P}$ w.r.t.~$\widetilde d_B$.
\end{lemma}

\begin{proof}
Boundedness, additivity and covariance of $w_P(B)$ hold by the same arguments as in the proof of Lemma~\ref{weight-function}. For almost invariance of $w_P(B)$ with respect to $\widetilde d_B$, it suffices to show:
\begin{quote}
For every $\varepsilon>0$ there exist $\delta>0$ and $U_0>0$ such that for all $U\ge U_0$ the following holds: If $B\in\mathcal{B}(U)$, a point set $P\in\mathcal{P}$ and a rotation $r_B$ are given such that $d(r_B)<\delta$, then $|w_{r_BP}(B)-w_P(B)|<\varepsilon\,\mathrm{vol}(B)$.
\end{quote}
Indeed, almost invariance then follows with the triangle inequality by a $3\varepsilon$--argument, together with almost invariance of $w_P(B)$ with respect to $d_B$, which is true by Lemma~\ref{weight-function}.

We obtain an estimate uniform in $P\in \mathcal{P}$ as
\begin{displaymath}
\begin{split}
|w_{r_BP}(B)&-w_P(B)|=\Big|\int_B{\rm d}x f(x^{-1}r_BP)-\int_B{\rm d}x f(x^{-1}P)\Big|\\
=&\,\Big|\int_{r_B^{-1}B}{\rm d}x f\circ r_B(x^{-1}P)-\int_B{\rm d}x f(x^{-1}P)\Big|\\
\le& \,\Big|\int_{r_B^{-1}B}{\rm d}x f\circ r_B(x^{-1}P)-\int_{r_B^{-1}B}{\rm d}x f(x^{-1}P)\Big|\\
&\qquad+\Big|\int_{r_B^{-1}B}{\rm d}x f(x^{-1}P)-\int_B{\rm d}x f(x^{-1}P)\Big|\\
\le& \, ||f\circ r_B-f||_\infty \cdot\mathrm{vol}(B)+||f||_\infty\cdot\mathrm{vol}((r_B^{-1}B)\Delta B).
\end{split}
\end{displaymath}
Here $\Delta$ denotes the symmetric difference. The second equation is a consequence of rotation invariance of the Lebesgue measure. The remaining inequalities rest on standard estimates and rotation invariance.

We restrict w.l.o.g.~to squarish boxes centered at the origin, which is possible due to $d(r_{B})=d(r_{tB})$ and $w_{r_BP}(B)=w_{r_{tB}tP}(tB)$ for $t\in\mathbb R^\mathsf{d}$ due to translation invariance of the Lebesgue measure. 
To analyse the first term, note that due to continuity of $f$, the map $(P,r_0)\mapsto f(r_0P)$ is uniformly continuous on the compact set $\mathcal{P}\times O(\mathsf{d})$, where we use the product topology. Noting that the maximum metric from the factors is compatible with the product topology, we conclude that the supremum norm $||f\circ r_0-f||_\infty$ gets arbitrarily small for $r_0\in O(\mathsf{d})$ sufficiently close to the identity.

For the second term, one may restrict analysis to $D\in\mathcal B(1)$. Indeed, if $B=\lambda D$ for some $\lambda>0$, we have
\begin{displaymath}
\mathrm{vol}((r_B^{-1}B)\Delta B) = \lambda^{\mathsf{d}} \mathrm{vol}((r_{D}^{-1}D)\Delta D)=\frac{\mathrm{vol}((r_{D}^{-1}D)\Delta D)}{\mathrm{vol}(D)}\mathrm{vol}(B).
\end{displaymath}
We have to show $\mathrm{vol}((r_{D}^{-1}D)\Delta D)\to0$ as $d(r_D)\to0$, uniformly in $D\in\mathcal B(1)$. Due to linearity of the volume in the coordinate directions, it suffices to consider the squarish box $D\in\mathcal B(1)$ of maximal side lengths $2$. The claim follows if $\mathrm{vol}((r_{D}D)\cap D)\to\mathrm{vol}(D)$ as $d(r_D)\to0$. But the latter can be shown using dominated convergence, by noting that $1_{r_DD}\to 1_D$ as $d(r_D)\to0$ pointwise on $(\partial D)^c$, together with $\mathrm{vol}(\partial D)=0$. Here $\partial A$ denotes the topological boundary of $A$.
\end{proof}

The proof of the following theorem is analogous to that of Proposition~\ref{thm:linarep-uniqerg}. In the present context, it rests on Lemma~\ref{lem:wigwf2} together with Proposition~\ref{mainlemma} iii), which remains valid for linearly wiggle--repetitive point sets and weight functions w.r.t.~$\widetilde d_B$.

\begin{theorem}\label{thm:alwruni}
Let $\widehat P\in \mathcal{P}_r$ be almost linearly wiggle--repetitive. Then $X_{\widehat P}$ is uniquely ergodic w.r.t.~$T=\mathbb R^{\mathsf d}$. \qed
\end{theorem}

Unique ergodicity also holds  w.r.t.~the Euclidean group $T=E(\mathsf{d})$.

\begin{lemma}\label{lem:wigwf}
For $\widehat P\in \mathcal{P}_r$ consider $\mathcal{P}:=X_{\widehat P}$. Then for every $f\in C(\mathcal{P})$ the function
\begin{displaymath}
(B,P)\mapsto w_P(B):=\int_{B\times O(\mathsf{d})}{\rm d}x f(x^{-1}P)
\end{displaymath}
is a weight function on $\mathcal B\times \mathcal{P}$ w.r.t.~$\widetilde d_B$.
\end{lemma}

\begin{proof}[Sketch of Proof]
As the arguments are similar to the case of $T=\mathbb R^{\mathsf d}$, we only describe the modifications. We can estimate
\begin{displaymath}
\begin{split}
|w_{r_BP}(B)&-w_P(B)|=\,\Big|\int_{B\times O(\mathsf{d})}{\rm d}x f(x^{-1}r_BP)-\int_{B\times O(\mathsf{d})}{\rm d}x f(x^{-1}P)\Big|\\
=&\,\Big|\int_{r_B^{-1}(B\times O(\mathsf{d}))}{\rm d}x f(x^{-1}P)-\int_{B\times O(\mathsf{d})}{\rm d}x f(x^{-1}P)\Big|\\
\le& \, ||f||_\infty\cdot\mathrm{vol}((r_B^{-1}(B\times O(\mathsf{d})))\Delta (B\times O(\mathsf{d})))\\
= &\,  ||f||_\infty\cdot\mathrm{vol}((r_B^{-1}B)\Delta B)\cdot\mathrm{vol}(O(\mathsf{d})),
\end{split}
\end{displaymath}
where used left invariance of the Haar measure in the first equation. Now one can argue as in the proof of Lemma~\ref{lem:wigwf2}.
\end{proof}

The proof of the following result is analogous to the proof of Theorem~\ref{thm:alruni}.

\begin{theorem}\label{thm:alwrunied}
Let $\widehat P\in \mathcal{P}_r$ be almost linearly wiggle--repetitive. Then $X_{\widehat P}$ is uniquely ergodic w.r.t.~the Euclidean group $T=E(\mathsf{d})$.\qed
\end{theorem}

\section{Dynamical properties of substitution tilings}\label{sec:dpst}

As an application of the previous results, we consider dynamical systems associated to the tiling spaces $X_\sigma$ of Section~\ref{chap:subs}. 
Whereas Theorem~\ref{unicirc} includes known results such as \cite{R95b}, \cite[Thm~3.1]{So97}, and \cite[Prop~3.1]{FS3}, our proof is 
 alternative to the previous approaches.

Consider the topology $LMT_G$ on $X_\sigma$ where two tilings are close, if they agree -- after some 
shift in $G$ close to the identity -- on a large ball about the origin. It is generated by a metric as in Section~\ref{sec1}. The group 
$G$ has a canonical left action on $X_\sigma$, which is continuous w.r.t.~$LMT_G$, 
compare \cite{bsj, So97} for $G=\mathbb R^\mathsf{d}$.  If $\sigma$ has FLC w.r.t.~$G$, then $X_\sigma$ is compact by standard 
reasoning,  see e.g.~\cite{RaWo92}.  

Point sets may be constructed from tilings as follows.
A \emph{prototile decoration} $\Phi$ on $(\mathcal F,G)$ is given by finite sets $\Phi(\{S_1\}),\ldots, \Phi(\{S_m\})$ in $\mathbb R^\mathsf{d}$ such that 
$\Phi(\{S_i\})\subset \mathrm{supp}(S_i)$ for all $i$. A prototile decoration $\Phi$ on $(\mathcal F,G)$ is \emph{structure--preserving} if
there is a natural extension to a map $\Phi:\mathcal C_{(\mathcal F,G)}\to\mathcal{P}_r$ for some $r>0$,  which is one-to-one. Here we call an extension of a prototile decoration $\Phi$ natural if $\Phi(\{S\})=g\Phi(\{S_i\})$ 
for tiles $S=gS_i$ and if $\Phi(\mathcal C)=\bigcup_{S\in \mathcal C}\Phi(\{S\})$ for $(\mathcal{F},G)$-packings $\mathcal C$.
If $\Phi$ is a structure--preserving prototile decoration, then $\mathcal T$ and $\Phi(\mathcal T)$ are mutually locally derivable (MLD) for every tiling $\mathcal T\in \mathcal C_{(\mathcal F,G)}$, see \cite[Def~2.5]{FS07} for the concept of MLD in our situation.

\begin{proposition}\label{prop:deco}
Let $\sigma$ be a substitution on $(\mathcal F,G)$, such that any prototile has a finite symmetry group centered in the interior of the prototile. Then there exists a structure--preserving prototile decoration on $(\mathcal F,G)$.
\end{proposition}

\begin{proof}
For $i\in\{1,\ldots,m\}$, denote  by $H_i\subset G$ the finite symmetry group $\{g\in G\, |\, gS_i=S_i\}$ of $S_i$ and choose a fixed point $y_i\in\mathrm{int}(S_i)$ common to all $g\in H_i$. With $e_1$ the unit vector in the first coordinate direction, consider the finite point set
\begin{displaymath}
P_i:=\{y_i\}\cup \frac{i}{2m+1} H_i (y_i+e_1) \cup i H_i (y_i+e_1).
\end{displaymath}
Clearly $P_i$ has the same symmetry group as $S_i$, and it is possible to recover $y_i$, $i$ and $e_1$ from $P_i$. 
Next, choose $D>0$ such that $B_D(y_i)\subset \mathrm{supp}(S_i)$ for all $i$ and define the prototile decoration
\begin{displaymath}
\Phi(\{S_i\}):=\frac{D}{2m} P_i \subset \mathrm{supp}(S_i).
\end{displaymath}
The natural extension to tiles is well--defined, as it respects the tile symmetries, i.e., $g\Phi(\{S_i\})=\Phi(\{S_i\})$ for $gS_i=S_i$. 
Denote by $r>0$ a radius of discreteness common to all $\Phi(\{S_i\})$. The natural extension to packings $\Phi:\mathcal C_{(\mathcal F,G)}\to\mathcal{P}_r$ 
is well--defined, since any two images of different tiles in a packing have a distance larger than $r$. The map $\Phi$ is indeed one-to-one, since it is one-to-one on tiles, and since the diameters of the tile images are chosen small enough, such that every packing image can be partitioned into tile images, and since this partition can be reconstructed from the packing image.
\end{proof}

\begin{remark}
The condition of the above proposition is satisfied for any substitution with convex prototiles and finite symmetry groups, as can be inferred from the proof of \cite[Thm~2.4]{M72}. It may also be satisfied if the convexity constraint is somewhat relaxed at the boundary of the prototiles. A simple example of a tiling with a non-convex prototile, which admits a structure--preserving prototile decoration, is the chair tiling.
\end{remark}

\begin{theorem}\label{unicirc}
Fix $G=\mathbb R^\mathsf{d}\rtimes H$, with $H$ a subgroup of $O(\mathsf{d})$. Let $\sigma$ be a primitive substitution on $(\mathcal F,G)$, which admits a structure--preserving prototile decoration. Assume that $X_\sigma$ is non--empty and equipped with the topology $LMT_G$. Assume that $\sigma$ has FLC w.r.t.~$G$ and DTO w.r.t.~$H$. Then for every group $F$ satisfying $\mathbb R^\mathsf{d}\subseteq F\subseteq G$, the dynamical system $(X_\sigma,F)$ is minimal and uniquely ergodic.
\end{theorem}

\begin{remark}
A simple example is the chair tiling with $G=\mathbb R^\mathsf{2}\rtimes D_4$. Here $LMT_G$ equals $LMT_{\mathbb R^\mathsf{2}}$, since $D_4$ is discrete.
Other examples are  the pinwheel tiling with $G=E(\mathsf{2})$ and the quaquaversal tiling with $G=\mathbb R^\mathsf{3}\rtimes SO(\mathsf{3})$.
\end{remark}

\begin{proof}
Fix a structure--preserving prototile decoration $\Phi$ on $(\mathcal F,G)$.  Take any $\mathcal T\in X_\sigma$ and define $P:=\Phi(\mathcal T)$. Then $P$ inherits (almost) linear wiggle--repetitivity from $\mathcal T$, since the chosen metrics on $O(\mathsf{d})$ are equivalent. In particular, $P$ is almost repetitive w.r.t.~$\mathbb R^\mathsf{d}$ by Remark \ref{rem:alwr}. 
Consider $X_P:=\overline{\{xP|x\in \mathbb R^\mathsf{d}\}}^{LMT_G}$ and note that $LMT_G$ equals $LRT$, as $P$ inherits FLC from $\mathcal T$. We conclude that $(X_P,\mathbb R^\mathsf{d})$ is minimal by  Lemma~\ref{lem:ar} and Theorem~\ref{charmin}, and that it is uniquely ergodic by Theorem~\ref{thm:alwruni}. 
Now consider $X_\mathcal{T}:=\overline{\{x\mathcal T| x\in \mathbb R^\mathsf{d}\}}^{LMT_G}$. Since $\Phi$ is structure--preserving, the restriction $\Phi: X_\mathcal{T}\to X_P$ is a homeomorphism satisfying $\Phi(g \mathcal T')=g\Phi(\mathcal T')$. Hence, by topological conjugacy, the dynamical system $(X_\mathcal{T},\mathbb R^\mathsf{d})$ is minimal and uniquely ergodic, too. But $X_\mathcal{T}=X_\sigma$ due to DTO. Thus $(X_\sigma, \mathbb R^\mathsf{d})$ is minimal and uniquely ergodic.

Now fix a group $F$ satisfying $\mathbb R^\mathsf{d}\subseteq F\subseteq G$. Since $X_\sigma$ is compact, $(X_\sigma,F)$ possesses at least one 
$F$--invariant ergodic probability measure, compare the proof of \cite[Cor~6.9.1]{W82} for $\mathbb Z$--actions and the proof of \cite[Theo~2.16]{MR}. 
Two such measures $\mu,\nu$ are in particular $\mathbb R^\mathsf{d}$--invariant. Hence $\mu=\nu$ by unique ergodicity w.r.t.~$\mathbb R^\mathsf{d}$, and $(X_\sigma,F)$ is uniquely ergodic.
For every $\mathcal T'\in X_\sigma$ its $F$--orbit is dense in $X_\sigma$ by minimality of $(X_\sigma, \mathbb R^\mathsf{d})$. Hence $(X_\sigma, F)$ is minimal, too.
\end{proof}

\section*{Acknowledgements}

We gratefully acknowledge discussions with Alexey Garber at an initial stage of this project. 
Special thanks are due to Daniel Lenz for a thorough discussion on results of this paper. We also thank the anonymous referees for 
valuable comments, corrections and hints on the literature which have improved the article. This work is partly 
supported by  the German Science Foundation DFG within CRC701.  


\end{document}